\documentclass[reqno]{amsart}

\usepackage{amsrefs}
\usepackage{hyperref}
\usepackage[all]{xy}
\makeatletter

\newcommand{\xyR}[1]{
  \makeatletter
  \xydef@\xymatrixrowsep@{#1}
  \makeatother
}
\makeatletter
\newcommand{\xyC}[1]{
  \makeatletter
  \xydef@\xymatrixcolsep@{#1}
  \makeatother
}

\newtheorem{theo}{Theorem}[section]
\newtheorem{lemma}[theo]{Lemma}

\newtheorem{remark}[theo]{Remark}
\newtheorem{problem}[theo]{Problem}
\def\unnumthm{\subsubsection*{Theorem}}
\def\chl{\subsubsection*{Theorem (Curtis-Hedlund-Lyndon)}}
\def\bhm{\subsubsection*{Theorem (Blanchard-Host-Maass) }}
\def\spt{\subsubsection*{The Small Period Theorem (Coven-Hedlund-Rhodes)}}

\theoremstyle{definition}
\newtheorem{definition}[theo]{Definition}
\usepackage{graphicx}
\usepackage{pstricks, enumerate, pst-node, pst-text, pst-plot}

\newcommand{\intav}[1]{\mathchoice {\mathop{\vrule width 6pt height 3 pt depth  -2.5pt
\kern -8pt \intop}\nolimits_{\kern -6pt#1}} {\mathop{\vrule width
5pt height 3  pt depth -2.6pt \kern -6pt \intop}\nolimits_{#1}}
{\mathop{\vrule width 5pt height 3 pt depth -2.6pt \kern -6pt
\intop}\nolimits_{#1}} {\mathop{\vrule width 5pt height 3 pt depth
-2.6pt \kern -6pt \intop}\nolimits_{#1}}}

\newcommand{\intavl}[1]{\mathchoice {\mathop{\vrule width 6pt height 3 pt depth  -2.5pt
\kern -8pt \intop}\limits_{\kern -6pt#1}} {\mathop{\vrule width 5pt
height 3  pt depth -2.6pt \kern -6pt \intop}\nolimits_{#1}}
{\mathop{\vrule width 5pt height 3 pt depth -2.6pt \kern -6pt
\intop}\nolimits_{#1}} {\mathop{\vrule width 5pt height 3 pt depth
-2.6pt \kern -6pt \intop}\nolimits_{#1}}}

\begin{document}

\title[Large semigroups of cellular automata]{Large semigroups of cellular automata}

\author[Yair Hartman]{Yair Hartman}
\address{Weizmann Institute of Science, Faculty of Mathematics and Computer Science, POB 26, 76100, Rehovot, Israel.}
\email{yair.hartman@weizmann.ac.il}

\date{\today}

\begin{abstract}
In this article we consider semigroups of transformations of cellular automata which act on a fixed shift space.
In particular, we are interested in two properties of these semigroups which relate to {}``largeness''. The first property is ID and the
other property is maximal commutativity (MC). A semigroup has the ID property
if the only infinite invariant closed set (with respect to the semigroup
action) is the entire space. We shall consider two examples of semigroups: one is spanned by cellular automata transformations that represent multiplications by integers on the one-dimensional torus and the other one consists of all the cellular automata transformations which are
linear (when the symbols set is the ring of integers mod n). It will be shown that the two properties of these semigroups depend on the number of symbols s. The multiplication semigroup is ID and MC if and only if s is not a power of prime. The linear semigroup over the mentioned ring is always MC but is ID if and only if s is prime. When the symbol set is endowed with a finite field structure (when possible) the linear semigroup is both ID and MC. In addition, we associate with each semigroup which acts on
a one sided shift space a semigroup acting on a two sided shift space, and vice versa, in such a way that preserves the ID and the MC properties.
\end{abstract}

\maketitle

\section*{Introduction}

Consider the one dimensional torus $\mathbb{T}=\frac{\mathbb{R}}{\mathbb{Z}}$
and the maps $m_{u}\left(x\right)=\left(ux\right)_{mod\,1}$ for $u\in\mathbb{Z}$.
$\left(\mathbb{T},m_{u}\right)$ is a dynamical system which has many invariant
probability measures. The situation seems to be different if one considers
the multi-parameter dynamical system $\left(\mathbb{T},\Sigma\right)$
where $\Sigma=\left\{ m_{2}^{r}m_{3}^{s}\right\} _{r,s \in\mathbb{N}}$
(the semigroup spanned by both $m_{2}$ and $m_{3}$). In 1967, Furstenberg
asked the following question (Furstenberg's Conjecture): Is it true
that the only non-atomic ergodic $\left\{ m_{2}^{r}m_{3}^{s}\right\} $-invariant
measure is the Lebesgue measure? This conjecture has been extensively
studied, and yet it is still open in the general sense. Rudolph \cite{rudolph1990x2}
proved that any measure which is $\left\{ m_{2}^{r}m_{3}^{s}\right\} $-invariant
is a linear combination of the Lebesgue measure and measures with
zero entropy (with respect to $m_{2}$). For more developments around
the conjecture the reader is referred to \cite{lindenstrauss2005invariant}
or \cite{katok1996invariant}.

The topological version of the conjecture was proved by Furstenberg
in \cite{furstenberg1967disjointness}:
\unnumthm
\footnote{The original theorem is more general as it holds for non-lacunary
semigroups, as will be discussed later.%
}(Furstenberg) The only infinite closed subset of $\mathbb{T}$ which
is $\left\{ m_{2}^{r}m_{3}^{s}\right\} $-invariant is $\mathbb{T}$.
\medbreak

This theorem classifies the closed $\left\{ m_{2}^{r}m_{3}^{s}\right\} $-invariant
subsets of $\mathbb{T}$, or: the subsystems of the topological dynamical system $\left(\mathbb{T},\left\{m_2^{r},m_3^{s}\right\}\right)$. This property of $\left\{ m_{2}^{r}m_{3}^{s}\right\} $
acting on $\mathbb{T}$ leads to the following general definition:
A semigroup $\Sigma$ acting on a topological space $\Omega$ will be called
\textbf{ID} (stands for \textbf{I}nfinite is \textbf{D}ense) if the
only infinite closed $\Sigma$-invariant subset of $\Omega$ is $\Omega$
itself. In terms of subsystems, an action of an ID semigroup is almost minimal: the only (proper) subsystems are finite.

In \cite{berend1983multi}, Berend classified the ID semigroups
of endomorphisms of $\mathbb{T}^{n}$. Later on, in \cite{berend1984multi},
he gave a necessary and sufficient condition for a commutative semigroup
of endomorphisms of $G$ to be an ID semigroup, where $G$ is a finite
dimensional connected compact abelian group.

Let $\Omega$ be a full shift space, i.e. the space of all infinite one (or two) sided sequences over some finite symbol set and denote by $\sigma:\Omega\to\Omega$ the left shift. Endomorphisms of the dynamical system $\left(\Omega,\sigma\right)$
are continuous transformations $\tau:\Omega\to\Omega$
which satisfy $\tau\sigma=\sigma\tau$. These transformations are called Cellular Automata Transformations. In this article the term {}``transformation'' will always refer to a transformation of a cellular automaton. The object studied here is a semigroup which consists of such transformations.

Note that any transformation maps each periodic sequence to a periodic sequence (the period may decrease). Therefore, a semigroup of transformations cannot have minimal action. The ID property is a natural generalization of the minimality concept for symbolic dynamics.

The commutativity property is also of interest.
A semigroup is said to be maximal commutative (MC) if it is commutative and it contains all the transformations that commute with all its members.

In section \ref{sec:Arithmetic-semi-groups} we recall the definition of {}``multiplication
cellular automaton'', studied by Blanchard, Host and Maass (\cite{0881.58035} and \cite{blanchard1997dynamical}), which represents a multiplication by an integer on
the one dimensional torus. We confirm that Furstenberg's result can be applied to semigroups spanned by transformations of multiplication cellular automata when the number of symbols is not a power of a prime. Then, we show that the semigroup of all multiplication transformations is MC if and only if the number of symbols is not a power of a prime.

In section \ref{sec:Algebraic-Semigroups} another semigroup, consisting of transformations which are all linear, is considered. First, we prove using {}``symbolic tools'' that the semigroup which consists of all the transformations over a one sided shift space is ID, for a symbol set of any size.
When $s$ is a power of a prime, and we endow the symbol set with a finite field structure, a generalization of this proof will show that the the semigroup of linear transformations is also ID. When considering the symbol set of $s$ symbols as the ring $\frac{\mathbb{Z}}{s\mathbb{Z}}$, the semigroup of linear transformations is ID if and only if $s$ is a prime.
However, for any $s$, when considering the symbol set as a ring or as a field, the semigroup of linear transformations is always MC. Study of transformations which are linear when the symbol set is the ring $\frac{\mathbb{Z}}{s\mathbb{Z}}$ can be found in \cite{ito1983linear}. Transformations which are linear, when the symbol set is a field (in a more general settings), were considered in a series of papers of Ceccherini-Silberstein and Coornaert starting from \cite{ceccherini2006garden} and recently appeared in chapter 8 of their book \cite{ceccherini2010cellular}.

Given a transformation, the question which other transformations commute with it, is known as {}``The Commuting Block Maps Problem'' and was introduced by Coven, Hedlund and Rhodes in \cite{coven1979commuting}. They presented a solution for some classes of block maps over the space $\left\{ 0,1\right\} ^{\mathbb{Z}}$.
One of the classes is the class of homogeneous block maps, which
are called here {}``linear cellular automata transformations''.
They proved that when the symbol set is $\left\{ 0,1\right\} $, then the semigroup of all linear cellular automata transformations is MC. The methods used in \cite{coven1979commuting} can not be easily generalized since they use specific properties related to the ring $\frac{\mathbb{Z}}{2\mathbb{Z}}$.
The results in sections \ref{sec:Arithmetic-semi-groups}, \ref{sec:Algebraic-Semigroups} can be viewed as the solution of the commuting problem for the classes of multiplications transformations and the linear transformations.

Sections 1-4 treat semigroups acting on a one sided shift space. Two sided shift spaces and semigroups of transformations acting on them, will be introduced in section \ref{sec:Two-sided-shift}. Given a semigroup acting on a one sided shift space, we construct a semigroup acting on the related two sided shift space and vice versa.
We prove that the MC and the ID properties are preserved under these constructions.

Boyle, Lind and Rudolph in \cite{boyle1988automorphism}, proved that the automorphism group of a two-sided subshift of finite type is an ID semigroup, for any finite symbol set. In section \ref{sec:The-Action-of-Aut} we present a corollary of
the theorems proved in sections \ref{sec:Arithmetic-semi-groups}
and \ref{sec:Two-sided-shift}. This corollary yields a stronger version of Boyle-Lind-Rudolph Theorem (only for the full shift space), whenever the number of symbols is not a power of a prime.

Since both properties considered in this article are related to {}``largeness'', we investigate, in section \ref{sec:ID-and-Maximal}, the question how {}``small'' (in term of number of generators) can be a semigroup having these properties.

\section{\label{sec:Introduction}Preliminaries}

Define $\Lambda_{s}$ to be a finite set of symbols such that $\left|\Lambda_{s}\right|=s$ for $s\ge2$ ($s\in \mathbb{N}$).
Throughout this article, we will assume that the symbol set is $\Lambda_{s}=\left\{ 0,1,...,s-1\right\} $.
One sided (full) shift space is the topological dynamical system
$\left(\Lambda_{s}^{\mathbb{N}},\sigma\right)$, where $\sigma$ is
the left shift: $\sigma\left(a\right)_{k}=a_{k+1}$
and $\Lambda_{s}^{\mathbb{N}}$ is the metric space of all one sided
sequences over $\Lambda_{s}$, equipped with the metric: \[
d\left(a,b \right)=\frac{1}{k+1}\]
 where $a_{k}\ne b_{k}$ and $a_{i}=b_{i}$ for all $i<k$.
If one treats $\Lambda_{s}$ as a discrete topological space, this metric induces the Tychonoff topology on the product $\Lambda_{s}^{\mathbb{N}}$.
Hence, $\Lambda_{s}^{\mathbb{N}}$ is a metric compact (Hausdorff) topological space.

A Cellular Automaton Transformation is a continuous function $\tau:\Lambda_{s}^{\mathbb{N}}\to\Lambda_{s}^{\mathbb{N}}$ which satisfies $\tau\sigma=\sigma\tau$. As stated before, a {}``transformation'' in this article we will always mean a cellular automaton transformation.
Denote the set of all the transformations over $\Lambda_{s}^{\mathbb{N}}$
by $CAT\left(\Lambda_{s}^{\mathbb{N}}\right)$. Note that $CAT\left(\Lambda_{s}^{\mathbb{N}}\right)$
is closed under the composition operation and therefore $CAT\left(\Lambda_{s}^{\mathbb{N}}\right)$
has a structure of a non-commutative semigroup.

These transformations were characterized in \cite{hedlund1969endomorphisms} as follows:
\chl
Let $\tau:\Lambda_{s}^{\mathbb{N}}\to\Lambda_{s}^{\mathbb{N}}$, then
$\tau\in CAT\left(\Lambda_{s}^{\mathbb{N}}\right)$,
if and only if there exists some $r\in\mathbb{N}$ and a function $f_{\tau}:\Lambda_{s}^{r+1}\to\Lambda_{s}$, such that \[
\tau\left(a\right)_{k}=f_{\tau}\left(a_{k}a_{k+1}...a_{k+r}\right)\]
 for any $a \in\Lambda_{s}^{\mathbb{N}}$ and
$k\in\mathbb{N}$.
$f_{\tau}$ is called a block map and $r$ is the radius of $\tau$. It is often convenient to use the same notation for the block map $f_{\tau}$ and the transformation $\tau$.
\medbreak

The following two properties will be discussed here.

\begin{definition}
A semigroup $\Sigma\subset CAT\left(\Lambda_{s}^{\mathbb{N}}\right)$
has the \textbf{ID property} if every closed $\Sigma$-invariant proper
subset of $\Lambda_{s}^{\mathbb{N}}$ is finite. In this case we say
that $\Sigma$ is an \textbf{ID semigroup} (ID stands for \textbf{{}``I}nfinite
is \textbf{D}ense'').
\end{definition}

\begin{definition}
A semigroup $\Sigma\subset CAT\left(\Lambda_{s}^{\mathbb{N}}\right)$
is \textbf{maximal commutative (MC)} if:
\begin{enumerate}
\item $\Sigma$ is commutative
\item for each $\mu\in CAT\left(\Lambda_{s}^{\mathbb{N}}\right)$ such that
$\mu\tau=\tau\mu$ (for all $\tau\in\Sigma$), $\mu$ necessarily belongs to $\Sigma$.
\end{enumerate}
\end{definition}

Throughout this article we denote $a_{\left[i,j\right)}=a_{i}a_{i+1}...a_{j-1}$
and $a_{\left[i,j\right]}=a_{i}a_{i+1}...a_{j}$ for $a \in\Lambda_{s}^{\mathbb{N}}$
or $a \in\Lambda_{s}^{\mathbb{Z}}$.

Let $\tau_1,\tau_2,\dots ,\tau_n$ be transformations. We denote by $\left\langle \tau_1 ,\tau_2,\dots ,\tau_n \right\rangle $ the semigroup generated by $\tau_1,\tau_2,\dots, \tau_n$, under the composition operation.
If $\Sigma=\left\{ \mu_{\alpha} \right\}_{\alpha \in I}$ is a semigroup of transformations (closed under composition) and $\tau$ is a transformation, then $\left\langle \Sigma ,\tau \right\rangle$ is the semigroup generated by $\tau$ and $\mu_\alpha$ for all $\alpha \in I$.

\section{\label{sec:Arithmetic-semi-groups}Multiplication Cellular Automata Transformations}

Let $a\in\Lambda_{s}^{\mathbb{N}}$. We can think
of $a$ as an $s$-representation of a point
in $\mathbb{T}=\frac{\mathbb{R}}{\mathbb{Z}}$ using the evaluation
function:

\[
V:\Lambda_{s}^{\mathbb{N}}\to\mathbb{T},\,\,\,\, V\left(a \right)=\sum\limits _{n=0}^{\infty}\frac{a_{n}}{s^{n+1}}\]
This function is injective up to the countable set
$\left\{ \frac{d}{p_{1}^{e_{1}}p_{2}^{e_{2}}...p_{r}^{e_{r}}}\right\} \subset\mathbb{T}$
where $p_{i}$ are primes satisfying $p_{i}|s$, $e_{i}\in\mathbb{N}$
 and $d\le {p_{1}^{e_{1}}p_{2}^{e_{2}}...p_{r}^{e_{r}}}$ ($d\in\mathbb{N}$).
Each point of this set has
a double representation:
\[
a_{0}a_{1}...a_{k}b000...\mbox{ and }a_{0}a_{1}...a_{k}\left(b-1\right)_{mod\, s}\left(s-1\right)\left(s-1\right)\left(s-1\right)...\]
We call these an {}``upper representation'' and a {}``lower representation'',
respectively.

In order to investigate the relationship between cellular automata transformations
and torus functions, we need the following definitions:
\begin{definition}
Let $\tau\in CAT\left(\Lambda_{s}^{\mathbb{N}}\right)$ and let $f:\mathbb{T}\to\mathbb{T}$
be a torus function, \textbf{$\tau$ represents $f$ }if
the following diagram:

\[
\xymatrix{\Lambda_{s}^{\mathbb{N}}\ar[r]^{\tau}\ar[d]^{V} & \Lambda_{s}^{\mathbb{N}}\ar[d]^{V}\\
\mathbb{T}\ar[r]^{f} & \mathbb{T}}
\]

is commutative.
\end{definition}

For a prime $p$ satisfying $p|s$, let $\mu_{p}\in CAT\left(\Lambda_{s}^{\mathbb{N}}\right)$ be the cellular
automaton transformation defined by:

\[
\mu_{p}\left(a \right)_{k}=\left(p\cdot a_{k}+\left\lfloor \frac{p\cdot a_{k+1}}{s}\right\rfloor \right)_{mod\, s}\]

In addition, define the identity map $\mu_1\left(a\right)_k=a_k$, the mirror map $\mu_{-1}\left(a\right)_k=s-a_k$, and the zero map $\mu_0\left(a\right)_k=0$. For any $u\in \mathbb{Z}$ such that $u={-1}^{e_0}{p_1}^{e_1}\cdots {p_d}^{e_d}$ where $p_i|s$ for all $i=1,\dots,d$ ($e_0\in \left\{0,1 \right\}$), define $\mu_{u}={\mu_{-1}}^{e_0} \circ {\mu_{p_1}}^{e_1} \circ \dots \circ{\mu_{p_d}}^{e_d}$ .

When there is ambiguity regarding the space where the transformation act, we denote $\mu_{u}^{(s)}$ to emphasize that $\mu_{u}^{(s)}\in CAT\left(\Lambda_{s}^{\mathbb{N}}\right)$.

\begin{theo}
\label{thm:multi-is-ID}
Let $s=p_{1}^{e_{1}}\cdot\dots\cdot p_{d}^{e_{d}}$ be the
prime decomposition of $s$, then
$\left\langle \mu_{p_{i}},\sigma \right\rangle \subset CAT\left(\Lambda_{s}^{\mathbb{N}}\right)$
(for any $i$) is ID if and only if $s$ is not a power of prime.
\end{theo}
\begin{theo}
\label{thm:multi-is-MC}
Let $s=p_{1}^{e_{1}}\cdot\dots\cdot p_{d}^{e_{d}}$ be the prime decomposition of $s$, then $\left\langle \mu_{p_{1}},\dots,\right.$ $\left.\mu_{p_{d}},\mu_{0},\mu_{1}\right\rangle \subset CAT\left(\Lambda_{s}^{\mathbb{N}}\right)$
is maximal commutative if and only if $s$ is not a power of prime.
\end{theo}

Denote $m_{u}:\mathbb{T}\to\mathbb{T}$ by $m_{u}\left(x\right)=\left(u\cdot x\right)_{mod\,1}$, and observe the following relationship:
\begin{lemma}
\label{lem:mu2-represents-the} $\mu_{p}\in CAT\left(\Lambda_{s}^{\mathbb{N}}\right)$ represents $m_{p}$ (multiplication by $p$ on $\mathbb{T}$).
\end{lemma}

\begin{proof}
For a {}``finite'' sequence, i.e. a sequence for which there exists
some $k\in\mathbb{N}$ such that $\sigma^{k}\left(a\right)=000...$,
the commutativity of the diagram follows from the multiplication algorithm
of numbers represented by $s$-expansion. Clearly, the set of all
such {}``finite'' sequences is dense in $\Lambda_{s}^{\mathbb{N}}$.
The proof follows from the continuity of $V,m_{p},$ and $\mu_{p}$.
\end{proof}

This means in particular that when $s=10$, the multiplication by $2$ and by $5$ on the torus are representable. Note however that $m_3$ cannot be represented by a cellular automaton transformation in $CAT \left( \Lambda_{10}^\mathbb{N}\right)$ (see BHM Theorem in the sequel).

\subsection*{\label{sub:mu2-mu_5-is-non}Proof of Theorem \ref{thm:multi-is-ID}}
Consider first the case where $s$ is not a power of prime, and let $p$ be a prime factor of $s$. The proof simply follows from Furstenberg's Theorem stated in \cite{furstenberg1967disjointness}.
Furstenberg showed that a non-lacunary%
\footnote{A multiplicative semigroup $\Gamma\subset\mathbb{Z}$ is lacunary
if there exists $\gamma\in\mathbb{N}$ such that every member of $\Gamma^{+}=\Gamma\cap\mathbb{N}$
is a power of $\gamma$. Otherwise, $\Gamma$ is non-lacunary.%
} semigroup of integers acting on $\mathbb{T}$ by multiplication has
the property that the only closed infinite subset of $\mathbb{T}$
which is invariant (with respect to this semigroup action) is $\mathbb{T}$.
Note that the semigroup $\left\{ p^{k_1}s^{k_2}\right\}_{k_1,k_2\in\mathbb{N}}ê $ is non-lacunary
and therefore has this property.

Let $A\subset\Lambda_{s}^{\mathbb{N}}$ be a closed $\left\langle \mu_{p},\sigma\right\rangle $-invariant
proper subset of $\Lambda_{s}^{\mathbb{N}}$. We wish to show that
$\left|A\right|<\infty$. $\Lambda_{s}^{\mathbb{N}}$ is a compact
space, hence both $A$ and $V\left(A\right)$ are compact. In particular,
$V\left(A\right)\subset\mathbb{T}$ is a closed set. $\mu_{p}\left(A\right)\subset A$, so $V\left(\mu_{p}\left(A\right)\right)\subset V\left(A\right)$.
From $V\left(\mu_{p}\left(A\right)\right)=p\cdot V\left(A\right)$
it follows that $p\cdot V\left(A\right)\subset V\left(A\right)$,
hence $V\left(A\right)$ is $m_{p}$-invariant. The same holds for
$\sigma$ and $m_{s}$. $V\left(A\right)$ is a proper
subset of $\mathbb{T}$, thus $V\left(A\right)$ is closed and a $\left\{ p^{k_1}s^{k_2}\right\} $-invariant
proper subset of $\mathbb{T}$, hence by Furstenberg's Theorem it is
finite. Although $V$ is not injective, for every $x\in\mathbb{T}$,
$\left|V^{-1}\left(x\right)\right|\le2$, therefore $A$ is finite.

In case $s=p^{m}$ ($p$ prime), $\left\langle \mu_{p},\sigma\right\rangle =\left\langle \mu_{p}\right\rangle \subset CAT\left(\Lambda_{s}^{\mathbb{N}}\right)$
is not ID, since $\left\{ p^{k_1}\right\}_{k_1\in\mathbb{N}} $ is lacunary and we can
use the same example of Furstenberg given in \cite{furstenberg1967disjointness} as follows. Let
$x=\sum _{i=1}^{\infty}p^{-i^{2}}$, then $\left\{ p^{k_1}x\right\}_{k_1\in\mathbb{N}}$
has only the limit points $0,p^{-1},p^{-2},...$ .
The set $\left\{ \mu_{p}^{k_1}\left(V^{-1}x\right)\right\}_{k_1\in\mathbb{N}} \subset \mathbb{T} $ is infinite and not dense, therefore $\left\langle \mu_{p}\right\rangle $ is not ID.

Alternatively, we give here a symbolic construction:
If $s=p^{2}$, then observe that by the definition of $\mu_{p}$,

\[
\begin{array}{ccccc}
\mu_{p} &  &  &  & \mu_{p}\\
00\mapsto0 &  &  &  & 00\mapsto0\\
01\mapsto0 &  &  &  & 0p\mapsto1\\
10\mapsto p &  &  &  & p0\mapsto0\\
11\mapsto p &  &  &  & pp\mapsto1\end{array}\]
Any irrational number $x\in\mathbb{T}$, whose expansion $a \in\Lambda_{p^{2}}^{\mathbb{N}}$
includes only the digits $0,1$, has an infinite $\mu_p$-orbit (since $a$ represents an irrational number). This means that $\left\{\mu_p^k\left(a\right)\right\}_{k\in\mathbb{N}}$ is infinite but none of the sequences in it contain digits other than $\left\{ 0,1,p\right\}$ (Actually, for even $k$'s the sequences will contain only $0$ and $1$ and for odd $k$'s the only digits that appear are $0$ and $p$). In particular,
the digit $p+1$ does not appear in $\left\{\mu_p^k\left(a\right)\right\}_{k\in\mathbb{N}}$.
Therefore, $\overline{\left\{\mu_p^k\left(a\right)\right\}_{k\in\mathbb{N}}}$ is an infinite, $\mu_{p}$-invariant proper subset of $\Lambda_{p^{2}}^{\mathbb{N}}$, hence $\left\langle \mu_{p}\right\rangle$ is not ID.

For $s=p^{n}$ observe that $\mu_{p}\left(0,0\right),\mu_{p}\left(0,p^{k}\right),\mu_{p}\left(p^{k},0\right),\mu_{p}\left(p^{k},p^{k}\right)\in\left\{ 0,p^{k+1}\right\} $
for $k\in\left\{ 0,1,\dots,k-2\right\} $ and $\mu_{p}\left(0,0\right),\mu_{p}\left(0,p^{n-1}\right),\mu_{p}\left(p^{n-1},0\right),\mu_{p}\left(p^{n-1},p^{n-1}\right)\in\left\{ 0,1\right\} $.
The digit $p+1$ cannot appear in the orbit $\left\{\mu_p^k\left(a \right)\right\}_{k\in\mathbb{N}}$ if $a_{n}\in\left\{0,1\right\}$ for all $n$. \hfill $\Box$

\begin{remark}
The last proof shows a way to construct closed, $m_{u}$-invariant subsets of $\mathbb{T}$ (which asserts that $m_{u}$ is not ID).
Take a non eventually periodic%
\footnote{Definitions and details provided in the next section.%
}
sequence $a \in\Lambda_{u^{2}}^{\mathbb{N}}$
composed only of $\left\{ 0,1\right\} $, consider $\mu_{u}\in CAT\left(\Lambda_{u^{2}}^{\mathbb{N}}\right)$
and take $V^{-1}\left( \overline{\left\{\mu_{p}^{k}\left(a\right) \right\}_{k\in \mathbb{N}} }\right)\subset\mathbb{T}$.
\end{remark}

Let $\tau:\Lambda_{s}^{\mathbb{N}}\to\Lambda_{s}^{\mathbb{N}}$
be a cellular automaton transformation. We would like to determine under which conditions $\tau$ represents a continuous torus function. In other words, when can $\tilde{\tau}:\mathbb{T}\to\mathbb{T}$ be constructed using the diagram:

\[
\xymatrix{\Lambda_{s}^{\mathbb{N}}\ar[r]^{\tau}\ar[d]^{V} & \Lambda_{s}^{\mathbb{N}}\ar[d]^{V}\\
\mathbb{T}\ar@{..>}[r]^{\tilde{\tau}} & \mathbb{T}}
\]
where $V$ is the evaluation function.

For $\tilde{\tau}$ to be well defined, we must confirm that if $a,b \in V^{-1}\left(x\right)$,
then $V\tau\left(a\right)=V\tau\left(b\right)$.
\begin{lemma}
\label{lem:If-f-commutes-f-can-be-lifted} Let $p|s$ and $s\ne p^m$. If $\tau\in CAT\left(\Lambda_{s}^{\mathbb{N}}\right)$
commutes with $\mu_{p}$, then $\tilde{\tau}$ is well defined continuous torus function.\end{lemma}

\begin{proof}
For concreteness, we will focus on the case $s=10, p=2$. The general case easily follows.
Let $\tau$ be a transformation that commutes with $\mu_{2}$.
Denote $A=\left\{ \frac{d}{2^{a}5^{b}}\right\} \subset\mathbb{T}$,
the set of all the points of the torus with two decimal expansions.
For each $x\in A$, define $x^{u}, x^{l} \in\Lambda_{10}^{\mathbb{N}}$
to be the upper and lower decimal representations of $x$, respectively.
In other words, $x^{u}\ne x^{l}$ but $V\left(x^{u}\right)=V\left(x^{l}\right)$.
We need to show that $V\tau\left(x^{u}\right)=V\tau\left(x^{l}\right)$
for all $x\in A$. For $x\in\mathbb{T}\backslash A$, $x^u=x^l$ is the unique decimal expansion of $x$.

Define $\Delta=\left\{ \delta\in\mathbb{T}|\,\exists x\in A,\,\left|V\tau\left(x^{u}\right)-V\tau\left(x^{l}\right)\right|=\delta\right\} $.
Note that $0 \in \Delta$ since one can check for example, that $V\tau\left(000...\right)=V\tau\left(999...\right)$.
For each $\delta \in \Delta$, $x\in\mathbb{T}$ will be called \textit{a witness
of $\delta$} if $\left|V\tau\left(x^{u}\right)-V\tau\left(x^{l}\right)\right|=\delta$.
It will be sufficient to prove that $\Delta=\left\{ 0\right\} $. We claim that $\Delta\subset\mathbb{T}$ satisfies the conditions of Furstenberg's Theorem.

Let $\left\{ \delta_{n}\right\} \subset\Delta$ such that $\delta_{n}\to\delta\in\mathbb{T}$.
For each $\delta_{n}$, denote its witness by $x^{n}\in A$. Denote $x^{n,u}$
and $x^{n,l}$ the two decimal expansions of $x^{n}$, i.e. $\left|V\tau\left(x^{n,u}\right)-V\tau\left(x^{n,l}\right)\right|=\delta_{n}$.
Since $\left(x^{n,u},x^{n,l}\right)\subset\left(\Lambda_{10}^{\mathbb{N}}\right)^{2}$
and due to the compactness of $\Lambda_{10}^{\mathbb{N}}$ we can
assume without loss of generality that $\left(x^{n,u},x^{n,l}\right)\to\left(y,z\right)$.
$V$ is continuous and $V\left(x^{n,u}\right)=V\left(x^{n,l}\right)$
for all $n$, thus $V\left(y\right)=V\left(z\right)$. This means
that $y$ and $z$ are two representations of the same element of
the torus. The continuity of $V\tau$ implies $\left|V\tau\left(y\right)-V\tau\left(z\right)\right|=\delta$,
so $V\left(y\right)=V\left(z\right)$ is a witness of $\delta$. Therefore
$\delta\in\Delta$ and $\Delta$ is a closed set.

Next, we show that $\Delta\subset\mathbb{T}$ is $\left\{ 2^{k_1}10^{k_2}\right\}_{k_1,k_2\in\mathbb{N}} $-invariant.
Let $\delta\in\Delta$ and let $x\in A$ be one of its witnesses. Define $y=2x\in\mathbb{T}$.
Obviously, $y\in A$. For the two decimal expansions of $y$:
\begin{eqnarray*}
\left|V\tau\left(y^{u}\right)-V\tau\left(y^{l}\right)\right| & = & \left|V\tau\mu_{2}\left(x^{u}\right)-V\tau\mu_{2}\left(x^{l}\right)\right|=\\
\left|V\mu_{2}\tau\left(x^{u}\right)-V\mu_{2}\tau\left(x^{l}\right)\right| & = & \left|2\cdot V\tau\left(x^{u}\right)-2\cdot V\tau\left(x^{l}\right)\right|=2\delta\in\mathbb{T}\end{eqnarray*}
This means that $y\in A$ is in fact a witness for $2\delta=m_{2}\left(\delta\right)$,
hence $m_{2}\left(\Delta\right)\subset\Delta$. The same is valid
for multiplication by 10 (note that in this case it is not necessary for $\tau$
to commute with $\mu_{2}$, but only to be a cellular automaton transformation, so
by definition it commutes with $\sigma=\mu_{10}$).

To conclude, $\Delta\subset\mathbb{T}$ is closed and invariant under
the action of the non-lacunary semigroup $\left\{ 2^{k_1}10^{k_2}\right\} $.
By Furstenberg's Theorem, $\Delta$ must be either $\mathbb{T}$ or
finite. Recall that we need to show that $\Delta=\left\{ 0\right\} $.

Assume by contradiction that there exists $\delta\in\Delta$ such that $\delta\ne 0$ and let
$x$ be a witness of $\delta$. Denote $x^{l}=a_{0}a_{1}...a_{k}b999...$,
$x^{u}=a_{0}a_{1}...a_{k}\left(b-1\right)_{mod\,10}000...$.
Then $\left(\frac{x}{10}\right)^{l}=0a_{0}a_{1}...a_{k}b999...$,
$\left(\frac{x}{10}\right)^{u}=0a_{0}a_{1}...a_{k}\left(b-1\right)_{mod\,10}000...$.
If $r\le k+1$ ($r$ is the radius of $\tau$), then $\left|V\tau\left(\left(\frac{x}{10}\right)^{u}\right)-V\tau\left(\left(\frac{x}{10}\right)^{l}\right)\right|=\frac{\delta}{10}\in\Delta$.
The same argument for $\frac{x}{10^{2}}$ implies that $\frac{\delta}{10^{2}}\in\Delta$, similarly $\frac{x}{10^{j}}$ implies that $\frac{\delta}{10^{j}}\in\Delta$ for all $j\in\mathbb{N}$, hence $\Delta$ is infinite.

If $r\ge k+2$, note that for every $c\in\mathbb{N}$, $\left|V\tau\left(\left(\frac{x}{10^{c}}\right)^{u}\right)-V\tau\left(\left(\frac{x}{10^{c}}\right)^{l}\right)\right|\ne0$,
so we can use the same procedure for $\frac{x}{10^{c}}$ with a sufficiently
large $c$, concluding again that $\Delta$ is infinite. (In this
case, the decimal expansions are: \[
\left(\frac{x}{10^{c}}\right)^{u}=0...0a_{1}...a_{k}b000...\mbox{ and }\left(\frac{x}{10^{c}}\right)^{l}=0...0a_{1}...a_{k}\left(b-1\right)_{mod\,10}999...\]
where the number of $0$ at the beginning of each expansion is $c$.
Choosing $c$ such that $r\le c+k+1$ reduces the case to
the preceding one.)

We showed that if there exists $\delta\in\Delta$, $\delta\ne 0$, then $\Delta$ contains
infinitely many elements, which implies $\Delta=\mathbb{T}$. But
this is a contradiction since $\Delta$ is a countable set.
Thus $\Delta=\left\{0\right\}$ and $\tilde{\tau}$ is well defined.
Standard compactness argument shows that $\tilde{\tau}$ is continuous.
\end{proof}

Until now we showed that every transformation that commutes with
$\mu_{p}$ represents some continuous torus function. The next theorem, due to Blanchard, Host and Maass \cite{0881.58035}, classifies the continuous torus functions that can be represented by a cellular automaton transformation, for general $s$.
Although our definition for {}``representing a torus function'' is somewhat different from that of \cite{0881.58035}, the same arguments and constructions which appear in \cite{0881.58035} can be applied here. They proved:
\bhm
\label{thm:representable-is-mul} If $\tau\in CAT\left(\Lambda_{s}^{\mathbb{N}}\right)$
represents a continuous torus function $\tilde{\tau}:\mathbb{T}\to\mathbb{T}$, then $\tilde{\tau}$ is of the form $\tilde{\tau}\left(x\right)=ux$ where $u\in\mathbb{Z}$ divides some positive power of $s$, or $\tilde{\tau}$ is a constant function $\tilde{\tau}\left(x\right)\equiv\frac{t}{s-1}$, where $t\in\left\{ 0,1,\dots,s-2\right\} $.
\medbreak

Combining Lemma \ref{lem:If-f-commutes-f-can-be-lifted} and BHM Theorem completes the proof of Theorem \ref{thm:multi-is-MC}:

\subsection*{Proof of Theorem \ref{thm:multi-is-MC}}

Where $s$ is not a power of prime, clearly, $\left\langle \mu_{p_{1}},...,\mu_{p_{d}},\mu_{0}\right\rangle$
is commutative since the multiplications on the torus, $m_{p_{1}},...,m_{p_{d}}$ and $m_{0}$ commute.

If $\tau\in CAT\left(\Lambda_{s}^{\mathbb{N}}\right)$ commutes with this semigroup, in particular it commutes with $\mu_{p_{1}}$ and by Lemma \ref{lem:If-f-commutes-f-can-be-lifted} , $\tau$ represents a continuous torus function $\tilde{\tau}$.
By BHM Theorem, $\tilde{\tau}=m_{u}$, where each prime $p|u$ is one of $p_{1},..,p_{d}$
or $\tilde{\tau}$ is a constant map $\tilde{\tau}\equiv\frac{t}{s-1}$
for some $t\in\left\{ 0,1,\dots s-2\right\} $. Consider the latter
case where $\tilde{\tau}\equiv\frac{t}{s-1}$. Note that a constant
map commutes with another map iff the image of the constant map is
a fixed point of the other. Thus, in order to commute with all $\mu_{p_{i}}$,
$\frac{t}{s-1}$ must be a fixed point of $m_{p_{i}}$ for all $i$.
Clearly, $0$ is a fixed point of all $m_{p_{i}}$ so we add $\mu_{0}$
to the semigroup in order to get maximality%
\footnote{Note that sometime 0 can be replaced by other digits that have this property: $\frac{1}{2}$
is fixed point of both $m_{3}$ and $m_{5}$. In $\Lambda_{15}$,
$\frac{1}{2}$ is represented by the constant sequence $77\dots$ .
Denoting $\tau\left(a\right)=77\dots$, we
get that $\left\langle \mu_{3},\mu_{5},\tau \right\rangle \subset CAT\left(\Lambda_{15}\right)$
is MC as well as $\left\langle \mu_{3},\mu_{5},\mu_{0} \right\rangle \subset CAT\left(\Lambda_{15}\right)$.
However, a commutative semigroup can have at most one constant map
and for the general statement we choose $\mu_{0}$.%
}.

We conclude that when $s$ is not a power of prime,  $\tau\in\left\langle \mu_{p_{1}},...,\mu_{p_{d}},\mu_{0}\right\rangle $ and hence the semigroup is MC.

Note that when $s=p^{m}$, the proof above does not hold, since if we define $\Delta\subset\mathbb{T}$ as in the proof of Lemma \ref{lem:If-f-commutes-f-can-be-lifted}
then $\Delta$ is $\left\{ p^{k_1}s^{k_2}\right\} $-invariant.
$\left\{ p^{k_1}s^{k_2}\right\} =\left\{ p^{k_1}\right\} \subset\mathbb{N}$
is lacunary and therefore, we cannot apply Furstenberg's Theorem to conclude that
that $\Delta=\left\{ 0\right\}$.

We will show now that when $s=p^{m}$ then the semigroup is actually not MC.
Assume $s=p^m$. Define the block map $f:\Lambda_p^{m}\to\Lambda_{p^m}$ by $f \left(a_0a_1\dots a_{m-1} \right) = p^{m-1}a_0+p^{m-2}a_1+\dots+pa_{m-2}+a_{m-1}$. $f$ defines a function $\phi:\Lambda_p^\mathbb{N} \to \Lambda_{p^m}^\mathbb{N}$ by $\phi\left(a\right)_k=f\left(a_{[mk,m(k+1))}\right)$.
$\phi$ interprets the first $m$-word as a $p$ expansion of an element in $p^m$ and assigns it as the first coordinate. The second $m$-word (the symbols $a_{m}a_{m+1}\dots a_{2m-1}$) is interpreted in the same way and the result is assigned to the second coordinate and so on.

This gives a topological conjugacy $$\phi : \left(\Lambda_p^\mathbb{N},\left\langle \mu_p^{(p)},\mu_{0}^{(p)}\right\rangle\right) \to \left(\Lambda_{p^m}^\mathbb{N},\left\langle \mu_p^{(p^m)},\mu_{0}^{(p^m)}\right\rangle\right),$$ which means that $\phi$ forms a homeomorphism between $\Lambda_p^\mathbb{N}$ and $\Lambda_{p^m}^\mathbb{N}$, and that there exists a bijection $\tau\mapsto \tilde{\tau}$ from $\left\langle \mu_p^{(p)},\mu_{0}^{(p)}\right\rangle$ to $\left\langle \mu_p^{(p^m)},\mu_{0}^{(p^m)}\right\rangle$ such that for any $\tau\in \left\langle \mu_{p}^{(p)},\mu_{0}^{(p)}\right\rangle $, $\phi\left(\tau\right) = \tilde{\tau}\left(\phi\right)$.
Explicitly, if $\tau=\left(\mu_p^{(p)}\right)^{k_1}\left(\mu_0^{(p)}\right)^{k_2}$ then
$\tilde{\tau}=\left(\mu_p^{(p^m)}\right)^{k_1}\left(\mu_0^{(p^m)}\right)^{k_2}$.
Any $\nu \in CAT\left(\Lambda_p^\mathbb{N}\right)$ that commutes with $\left\langle \mu_p^{(p)},\mu_{0}^{(p)}\right\rangle$ defines $\bar{\nu} \in CAT\left(\Lambda_{p^m}^\mathbb{N}\right)$ by $\bar{\nu}=\phi(\nu(\phi^{-1}))$. The topological conjugacy asserts that $\bar{\nu}$ commutes with $\left\langle \mu_p^{(p^m)},\mu_{0}^{(p^m)}\right\rangle$.
Hence, in order to show that $\left\langle \mu_p^{(p^m)},\mu_{0}^{(p^m)}\right\rangle \subset CAT\left(\Lambda_{p^m}^\mathbb{N}\right)$ is not MC, it is enough to show that $\left\langle \mu_p^{(p)},\mu_{0}^{(p)}\right\rangle =\left\langle \sigma ,\mu_{0}\right\rangle  \subset CAT\left(\Lambda_{p}^\mathbb{N}\right)$ is not MC.

Now, any transformation in $CAT\left(\Lambda_{p}^\mathbb{N}\right)$ with $00\dots\in\Lambda_{p}^{\mathbb{N}}$      as a fixed point commutes with $\left\langle \sigma ,\mu_{0}\right\rangle$, or equivalently, any block map that sends the word $00\dots0\mapsto 0$ commutes with both $\sigma$ and $\mu_0$ and this completes the proof.
\hfill $\Box$

\begin{remark}
$\mu_{-1}\in CAT\left(\Lambda_s^\mathbb{N}\right)$ is the {}``mirror map'' e.g. in the case of $s=10$ its block map is $0\leftrightarrow9,1\leftrightarrow8,...,5\leftrightarrow5$.
From BHM Theorem we have that a non constant transformation that represents a continuous torus function must be of the form $\mu_u$ where $u\in\mathbb{Z}$. In the case $s=10$ for example, $\mu_{-1}$ also commutes with both $\mu_2$ and $\mu_5$.
Note that for positive $u$, $\mu_u$  keeps the same representations: if $a\in \Lambda_{10}^\mathbb{N}$ is the upper representation of $x\in\mathbb{T}$ then $\mu_u\left(a\right)$ is the upper representation of $ux\in\mathbb{T}$, while $\mu_u$ for $u<0$ flips between the representations. Thus, although $m_{-1}$ and $m_0$ commute, $\mu_{-1}$ and $\mu_0$ do not commute.
\end{remark}

To summarize, we have shown a way to construct semigroups which are both MC and ID whenever $s\ne p^{m}$ for a prime $p$. The same construction, when $s=p^m$, is neither MC nor ID.

Semigroups of $CAT\left(\Lambda_{p^{m}}^{\mathbb{N}}\right)$ which are both MC and ID are described in the next section.

\section{\label{sec:Algebraic-Semigroups}Linear Cellular Automata Transformations}

In this section we consider $\Lambda_{s}$ as the ring $\frac{\mathbb{Z}}{s\mathbb{Z}}$.
Note that the ring's structure depends on the interpretation of each
symbol in $\Lambda_{s}$; every bijection between the symbol set $\Lambda_{s}$
and the elements of $\frac{\mathbb{Z}}{s\mathbb{Z}}$ defines
a different ring structure on $\Lambda_{s}$. For simplicity choose
the identity permutation and treat $\Lambda_{s}$ as $\frac{\mathbb{Z}}{s\mathbb{Z}}$.
Thus, $\Lambda_{s}^{\mathbb{N}}$ is endowed with a structure of a
$\Lambda_{s}$-module.

When $s=p^m$ is power of a prime, one can give the symbol set a field structure (again, the interpretation of the symbols affects the field structure). We will distinguish between these objects by the following notations: $\Lambda_s$ is the ring $\frac{\mathbb{Z}}{s\mathbb{Z}}$ (either when $s=p^m$ or not) and $\mathbb{F}_{p^m}$ is the field of $p^m$ elements. Thus, $\mathbb{F}_{p^m}^\mathbb{N}$ is endowed with a structure of an infinite vector space over $\mathbb{F}_{p^m}$.

Define $LCAT\left(\Lambda_{s}^{\mathbb{N}}\right)$ to be the subset
of $CAT\left(\Lambda_{s}^{\mathbb{N}}\right)$, which consists of
all the cellular automata transformations which are linear, i.e. $LCAT\left(\Lambda_{s}^{\mathbb{N}}\right)=CAT\left(\Lambda_{s}^{\mathbb{N}}\right)\cap Hom_{\Lambda_{s}}\left(\Lambda_{s}^{\mathbb{N}},\Lambda_{s}^{\mathbb{N}}\right)$.
Clearly, $LCAT\left(\Lambda_{s}^{\mathbb{N}}\right)$ is a semigroup
since both $CAT\left(\Lambda_{s}^{\mathbb{N}}\right)$ and $Hom_{\Lambda_{s}}\left(\Lambda_{s}^{\mathbb{N}},\Lambda_{s}^{\mathbb{N}}\right)$
are closed under composition.

Given $\tau\in LCAT\left(\Lambda_{s}^{\mathbb{N}}\right)$, denote
its block map by $f_{\tau}:\Lambda_{s}^{r+1}\to\Lambda_{s}$. The linearity
of $\tau$ imposes the following condition on $f_{\tau}$:
\[
\alpha f_{\tau}\left(a_{[0,r]}\right)+\beta f_{\tau}\left(b_{[0,r]}\right)=f_{\tau}\left(\alpha a_{[0,r]}+\beta b_{[0,r]}\right)\]
for any $\alpha,\beta\in\Lambda_{s}$
and $a_{[0,r]},b_{[0,r]}\in\Lambda_{s}^{r+1}$.

This means that $\tau\in LCAT\left(\Lambda_{s}^{\mathbb{N}}\right)$
iff its block map is a $\Lambda_{s}$-linear functional on $\Lambda_{s}^{r+1}$.
We can describe the action of $\tau\in LCAT\left(\Lambda_{s}^{\mathbb{N}}\right)$
by a {}``shift-polynomial'' in the manner described below.

Consider the standard basis $\left\{ e^{i}\right\}_{i=0}^{r} $ of $\Lambda_{s}^{r+1}$ and let $c_{i}=f_{\tau}\left(e^{i}\right)$. Let $a_{0}...a_{r}\in\Lambda_{s}^{r+1}$, then $f_{\tau}\left(a_{0}...a_{r}\right)=c_{0}\cdot a_{0}+c_{1}\cdot a_{1}+...+c_{r}\cdot a_{r}$.
Now associate with $\tau$ the polynomial $p_{\tau}\left(x\right)=c_{0}+c_{1}x+c_{2}x^{2}+...+c_{r}x^{r}$,
$p_{\tau}\in\Lambda_{s}\left[x\right]$. Observe that substituting
a cellular automaton transformation in some polynomial $p\in\Lambda_{s}\left[x\right]$
yields a cellular automaton transformation. In particular, substituting the shift $\sigma$ in $p_{\tau}$ yields $\tau$.

One can check that the ring $LCAT\left(\Lambda_{s}^{\mathbb{N}}\right)$
with composition as multiplication and point-wise addition, is isomorphic to $\Lambda_{s}\left[x\right]$ by $\tau\mapsto p_{\tau}$. This was shown in \cite{ferguson1962some} and discussed in \cite{coven1979commuting} for the case $s=2$. In particular, $LCAT\left(\Lambda_{s}^{\mathbb{N}}\right)$
is a commutative semigroup, since the polynomial ring is commutative.

In the same manner define $LCAT\left(\mathbb{F}_{p^m}^\mathbb{N}\right)$ to be all the transformations which are linear. The same isomorphism to the polynomial ring holds for $LCAT\left(\mathbb{F}_{p^m}^\mathbb{N}\right)$ and hence the commutativity property follows.

In this section we will prove the following:
\begin{theo} \label{thm:LCAT-over-field-is-ID}
The semigroup $LCAT\left(\mathbb{F}_{p^m}^{\mathbb{N}}\right)$ is an ID for any finite field $\mathbb{F}_{p^m}$.
\end{theo}

\begin{theo} \label{thm:LCAT-is-ID}
The semigroup $LCAT\left(\Lambda_{s}^{\mathbb{N}}\right)$ is an ID
semigroup if and only if $s$ is prime.
\end{theo}

\begin{theo} \label{thm:LCAT-over-field-is-MC}
The semigroup $LCAT\left(\mathbb{F}_{p^m}^\mathbb{N}\right)$ is maximal commutative for any finite field $\mathbb{F}_{p^m}$.
\end{theo}

\begin{theo} \label{thm:LCAT-is-MC}
The semigroup $LCAT\left(\Lambda_{s}^{\mathbb{N}}\right)$ is maximal commutative for any $s\ge2$.
\end{theo}

In order to prove that $LCAT\left(\mathbb{F}_{p^m}^{\mathbb{N}}\right)$ is an ID semigroup we first prove that $CAT\left(\Lambda_{s}^{\mathbb{N}}\right)$
is an ID semigroup for any $s$\footnote{Note that when considering $CAT\left(\Lambda_{s}^{\mathbb{N}}\right)$, the symbol set $\Lambda_s$ carries no structure.}. Then we show the property for the semigroup $LCAT\left(\mathbb{F}_{p^m}^{\mathbb{N}}\right)$.

Note that for $s\ne p^{m}$ this claim results from section \ref{sec:Arithmetic-semi-groups}
as follows: if $s=p_{1}^{e_{1}}\cdot\dots\cdot p_{d}^{e_{d}}$ ($d>1$),
then $\left\langle \mu_{p_{i}},\sigma\right\rangle $ is an ID semigroup
(for $i=1,2,\dots,d$). In particular, $CAT\left(\Lambda_{s}^{\mathbb{N}}\right)$
is also an ID semigroup, as it contains a sub-semigroup which is ID.

To prove that $CAT\left(\Lambda_{s}^{\mathbb{N}}\right)$ is ID for
any $s$, we will use the following definitions:
\begin{definition}
A sequence $a\in\Lambda_{s}^{\mathbb{N}}$
is a \textbf{rich sequence }if there exists $\left\{ n_{m}\right\} \subset\mathbb{N}$
with $n_{m}\to\infty$ such that for each $m\in\mathbb{N}$ there
exists $k\left(m\right)\in\mathbb{N}$ with the property $a_{[i,i+k)}\ne a_{[j,j+k)}$
for all $i\ne j<n_{m}$.
\end{definition}

\begin{definition}
A sequence $a\in\Lambda_{s}^{\mathbb{N}}$ is
a periodic sequence if for some $c>0$,  $a_{n}=a_{n+c}$
for all $n\in\mathbb{N}$. $a \in\Lambda_{s}^{\mathbb{N}}$
is \textbf{$\mathbf{c}$-periodic} if $c$ is the minimal integer
for which this holds. $c$ is called \textbf{the period of $ a $}.
\end{definition}

\begin{definition}
A sequence $a\in\Lambda_{s}^{\mathbb{N}}$
is a \textbf{$\mathbf{b,c}$-eventually periodic sequence }if $\sigma^{b}\left(a\right)$
is $c$-periodic, but $\sigma^{b-1}\left(a\right)$
is not periodic.
For simplicity the parameters $b,c$ will be omitted when possible.
\end{definition}

\begin{lemma}
\label{lem:A-sequence-rich-iff-not}A sequence $a\in\Lambda_{s}^{\mathbb{N}}$
is eventually periodic if and only if it is not rich.\end{lemma}

One can give a purely symbolic proof for this lemma. However, consider the evaluation function $V:\Lambda_{s}^{\mathbb{N}}\to\mathbb{T}$,
mentioned in section \ref{sec:Arithmetic-semi-groups}, the eventually periodic sequences are mapped to rational numbers while rich sequences are $s$-expansions of irrational numbers (for any $s$).

It follows from the definition that a rich sequence has an ample supply of words ($a_{[i,j]}$ is a word appearing in $a $). This will enable in Theorem \ref{thm:CAT-is-ID}, a construction of specific transformations showing that the image of the action of $CAT\left(\Lambda_{s}^{\mathbb{N}}\right)$ on any rich sequence is dense. On the other hand, by definition, eventually periodic sequences lake this property. However, the maximal amount of different words can be found in an eventually periodic sequence is described in the next lemma
\begin{lemma}
\label{lem:variaty-in-e.p}Let $a\in\Lambda_{s}^{\mathbb{N}}$
be $b,c$-eventually periodic, then there exists $k\in\mathbb{N}$
such that $a_{[i,i+k)}\ne a_{[j,j+k)}$ for all $i\ne j<b+c-1$\end{lemma}

Denote \[B_{b,c}=\left\{ a_{0}\dots a_{b-1}\overline{a_{b}\dots a_{b+c-1}}\right\} = \left\{ a_{0}\dots a_{b-1}a_{b}\dots a_{b+c-1}a_{b}\dots a_{b+c-1} \dots \right\} \subset\Lambda_{s}^{\mathbb{N}} \]
Note that the parameters $b,c$ are not minimal as they are in the
definition of a $b,c$-eventually periodic sequence. $B_{b,c}$ can
be defined alternatively by the subset of $\Lambda_{s}^{\mathbb{N}}$ consisting
of all $b',c'$-eventually periodic sequences with $c'|c$, $b'\le b$.
\begin{lemma}
For any $b,c$, $B_{b,c}$ is closed and $CAT\left(\Lambda_{s}^{\mathbb{N}}\right)$-invariant.
\end{lemma}

\begin{proof}
Since $\left|B_{b,c}\right|<\infty$, $B_{b,c}$ is a closed
set. Obviously, if $a$ is $b,c$-eventually
periodic then for any $\tau\in CAT\left(\Lambda_{s}^{\mathbb{N}}\right)$,
$\tau\left(a\right)$ is $b',c'$-eventually
periodic with $c'|c$, $b'\le b$, thus $\tau\left(a\right)\in B_{b,c}$
which shows the invariance.
\end{proof}

\begin{lemma}
\label{lem:A-is-union-if-Bsc}Let $A\subset\Lambda_{s}^{\mathbb{N}}$
be a $CAT\left(\Lambda_{s}^{\mathbb{N}}\right)$-invariant subset.
If $A$ contains a $b,c$-eventually periodic sequence, then $B_{b,c}\subset A$.
\end{lemma}
\begin{proof}
Let $a\in A$ be $b,c$-eventually periodic
and let $x\in B_{b,c}$. We will define a transformation $\tau \in CAT\left(\Lambda_{s}^{\mathbb{N}}\right)$ such that $\tau\left(a\right)=x $.
From Lemma \ref{lem:variaty-in-e.p} there exists $k\in\mathbb{N}$
such that $a_{[i,i+k)}\ne a_{[j,j+k)}$ for all $i,j<b+c$. Thus, the following block map is well defined:

\[
f\left(y_{0}y_{1}\dots y_{k}\right)=\begin{cases}
x_{i} & \mbox{if }y_{0} y_{1} \dots y_{k}=a_{[i,i+k)}\\
0 & \mbox{otherwise}\end{cases}\]
If $\tau$ is the transformation associated with $f$, then $\tau\left(a\right)=x $.
This can be done for any $x \in B_{b,c}$, thus
$B_{b,c}\subset A$.
\end{proof}

\begin{theo}
\label{thm:CAT-is-ID}
The semigroup $CAT\left(\Lambda_{s}^{\mathbb{N}}\right)$ is an ID
semigroup for any $s\ge2$.
\end{theo}
\begin{proof}
Let $A$ be a $CAT\left(\Lambda_{s}^{\mathbb{N}}\right)$-invariant
subset of $\Lambda_{s}^{\mathbb{N}}$.

Distinguish between two cases:
\begin{enumerate}
\item If $A$ contains a rich sequence we claim that $\overline{A}=\Lambda_{s}^{\mathbb{N}}$.
Let $a\in A$ be a rich sequence, $x\in\Lambda_{s}^{\mathbb{N}}$
and let $\epsilon>0$. Choose $n_{0}>\frac{1}{\epsilon}$. By the definition of a rich sequence, it is possible to find a $k\in\mathbb{N}$
 such $a_{[i,i+k)}\ne a_{[j,j+k)}$ for any $i\ne j<n_{0}$. Define a block map as follows:

\[
f\left(y_{0}y_{1}\dots y_{k}\right)=\begin{cases}
x_{i} & \mbox{if }y_{0}y_{1}\dots y_{k}=a_{[i,i+k)}\\
0 & \mbox{otherwise}\end{cases}\]
Note that for a transformation $\tau$ defined by this block map,
$\tau\left(a\right)_{[0,n_{0})}=x_{[0,n_{0})}$
and thus $d\left(\tau\left(a\right),x\right)<\epsilon$.
This holds for any $x \in\Lambda_{s}^{\mathbb{N}}$
and arbitrary $\epsilon$, therefore $\overline{A}=\Lambda_{s}^{\mathbb{N}}$.

\item Otherwise, $A$ contains only eventually periodic sequences.
Assume that $A$ is infinite, we need to show that $\overline{A}=\Lambda_{s}^{\mathbb{N}}$.
From Lemma \ref{lem:A-is-union-if-Bsc}, $A$ is a
union of $B_{b,c}$ sets. Since $\left|B_{b,c}\right|<\infty$,
 $A$ is a union of infinitely many $B_{b,c}$ sets. Thus, $A$
contains $B_{b,c}$ with arbitrary large $b+c$.

Let $x\in\Lambda_{s}^{\mathbb{N}}$ and $\epsilon>0$.
There exists $B_{b,c}\subset A$ with $b+c>\frac{1}{\epsilon}$. By
definition, $B_{b,c}$ contains all the $b+c$ prefixes, in particular,
it contains a sequence $a$ which coincide with
$x $ on the first $b+c$ coordinates: $a_{[0,b+c)}=x_{[0,b+c)}$.
Since this can be done for any $\epsilon>0$, we conclude that $\overline{A}=\Lambda_{s}^{\mathbb{N}}$.

\end{enumerate}
\end{proof}

Recall that the goal was to prove that $LCAT\left(\mathbb{F}_{p^m}^{\mathbb{N}}\right)$
is an ID semigroup. The method used in the preceding proof was building specific block maps. This was possible due to the variety of different words appearing in a rich sequences, or, in infinitely many eventually periodic sequences. To be able to build such a block map for a transformation in $LCAT\left(\mathbb{F}_{p^m}^{\mathbb{N}}\right)$
we must ensure that there are enough words, which are not only different but also linearly independent.
\begin{lemma}
\label{lem:rich-independed}Let $a\in \mathbb{F}_{p^m}^\mathbb{N} $ be a rich
sequence, then for any $n\in\mathbb{N}$ there exists a $k\left(n\right)$
such that the set of vectors $\left\{ a_{[i,i+k)}\right\} _{i=0}^{n-1}\subset\mathbb{F}_{p^m}^{k}$
is linearly independent.\end{lemma}
\begin{proof}
Assume by contradiction that there exists an $n_{0}\in\mathbb{N}$
such that for any $k$, $\left\{ a_{[i,i+k)}\right\} _{i=0}^{n_{0}-1}$
is linearly dependent. For any $k\in\mathbb{N}$, define $r\left(k\right)$
as the minimal integer such that $a_{[r,r+k)}$ is linearly dependent
on its predecessors ($\left\{ a_{[i,i+k)}\right\} _{i=0}^{r-1}$
are linearly independent). For any $k$, $r\left(k\right)\le n_{0}-1$ and thus
there exists a sequence $k_{l}\to\infty$ such that $r\left(k_{l}\right)=r_{0}$
is fixed.

Since the field $\mathbb{F}_{p^m}$ is finite, there is a finite number
of linear dependencies for $n_{0}$ vectors. Assume without
loss of generality that for any $k_{l}$ the $r_{0}^{\mbox{th}}$
vector depends on its predecessors by the same linear combination:
$a_{[r_{0},r_{0}+k_{l})}=\sum _{i=0}^{r_{0}-1}\beta_{i}a_{[i,i+k_{l})}$
for all $l\in\mathbb{N}$.

In particular, for each $j>r_{0}$, find $k_l>j$ and conclude that the $j^{\mbox{th}}$
coordinate satisfies $a_{j}=\sum _{i=j-r_{0}}^{j-1}\beta_{i}a_{i}$.
Therefore, for any $l$ such that $k_{l}>s^{r_{0}}$ we can find $i,j<k_{l}$
such that $a_{[i,i+r_{0})}=a_{[j,j+r_{0})}$ and thus $a_{i+r_{0}}=a_{j+r_{0}}$.
By induction, $a$ is eventually
periodic, which is a contradiction.

\end{proof}

\begin{lemma}
\label{lem:e.p.-independed}Let $a\in \mathbb{F}_{p^m}^\mathbb{N}$ be a $b,c$-eventually
periodic sequence, then there exists some $k\in\mathbb{N}$ such that
the set of vectors $\left\{ a_{[i,i+k)}\right\} _{i=0}^{b-1}\subset\mathbb{F}_{p^m}^{k}$ is linearly independent.\end{lemma}
\begin{proof}
From Lemma \ref{lem:variaty-in-e.p} there exists a $k_{0}$ such that
all the words $\left\{ a_{[i,i+k_{0})}\right\} _{i=0}^{b-1}$ are
different%
\footnote{We can assure that all the first $b+c-1$ $k_{0}$-words in $n $
are different. Here we need only the first $b-1$ words to be different.%
}. Clearly, this holds for any $k>k_{0}$. If for some $k>k_{0}$ these
vectors are linearly independent then we are done.

Otherwise, enlarge $k_{l}\to\infty$ and repeat the argument
used in the previous lemma to get $r_{0}\le b-1$ such that $a_{[r_{0},r_{0}+k_{l})}=\sum_{i=0}^{r_{0}-1}\beta_{i}a_{[i,i+k_{l})}$.
But then $\sigma^{r_{0}}\left( a \right)$ is periodic
in contradiction to the minimality of $b$.
\end{proof}

Now we can prove Theorem \ref{thm:LCAT-over-field-is-ID}:
\subsection*{Proof of Theorem \ref{thm:LCAT-over-field-is-ID}}
Let $A\subset\mathbb{F}_{p^m}^{\mathbb{N}}$ be an infinite $LCAT\left(\mathbb{F}_{p^m}^{\mathbb{N}}\right)$-invariant set. We need to show that $A$ is dense. Let $\epsilon>0$.

If $A$ contains a rich sequence $a$, by Lemma
\ref{lem:rich-independed}, for any $n_0$, there exists a $k$, such that the first $n_{0}$
$k$-words of $a$ are linearly independent.
Choose $n_{0}>\frac{1}{\epsilon}$ and denote $k=k\left(n_{0}\right)$.
Given $x\in\mathbb{F}_{p^m}^{\mathbb{N}}$, consider
a block map that assigns to the block $a_{\left[i,i+k\right)}$ the value $x_i$ for each $i=0,\dots ,r-1$. Note that due to the independence of $\left\{ a_{[i,i+k)}\right\} _{i=0}^{n_{0}-1}$, such block maps exist and they are linear functionals on $\mathbb{F}_{p^m}^{k}$. Hence,
a transformation defined by such a block map is indeed linear and thus is a member of $LCAT\left(\mathbb{F}_{p^m}^{\mathbb{N}}\right)$.
Obviously, $d\left(\tau\left(a\right),x\right)<\epsilon$
and hence, $\overline{A}=\mathbb{F}_{p^m}^{\mathbb{N}}$.

Otherwise, $A$ is an infinite set of eventually periodic sequences.
Using Lemma \ref{lem:e.p.-independed} and a similar argument we get
the following. Let $x\in\mathbb{F}_{p^m}^{\mathbb{N}}$ and let $a\in A$ be a $b,c$-eventually periodic sequence.
We can construct a block map that is associated with $\tau\in LCAT\left(\mathbb{F}_{p^m}^{\mathbb{N}}\right)$
such that $d\left(\tau\left(a\right),x\right)<\frac{1}{b}$.
Thus, if $A$ contains a $b,c$-eventually periodic with arbitrary
large $b$, then $\overline{A}=\mathbb{F}_{p^m}^{\mathbb{N}}$.

We still need to consider the case where $A$ is an infinite set of
eventually periodic sequences with bounded parameter $b$. In this
case parameter $c$ of the sequences in $A$ in unbounded. The shift
$\sigma$ is in $LCAT\left(\mathbb{F}_{p^m}^{\mathbb{N}}\right)$ and $A$
is $LCAT\left(\mathbb{F}_{p^m}^{\mathbb{N}}\right)$-invariant, therefore
$A$ contains infinitely many periodic sequences with arbitrary large
period $c$.
Let $x\in\mathbb{F}_{p^m}^{\mathbb{N}}$ and let $a\in A$ be a $c$-periodic sequence.
Repeating the argument in Lemma \ref{lem:e.p.-independed} we can
find a $k\in\mathbb{N}$
such that all the first $c-1$ $k$-words in $a$ are linearly independent (using the minimality of $c$). Thus it is possible to construct in the same way a $\tau \in LCAT\left(\mathbb{F}_{p^m}^{\mathbb{N}}\right)$, such that $d\left(\tau\left(a\right),x \right)<\frac{1}{c}$.
Since $A$ contains sequences with arbitrary large $c$, $\overline{A}=\mathbb{F}_{p^m}^{\mathbb{N}}$.
\hfill $\Box$

\subsection*{Proof of Theorem \ref{thm:LCAT-is-ID}}
Clearly, if $s$ is prime then $LCAT\left(\Lambda_s^\mathbb{N}\right) = LCAT\left(\mathbb{F}_s^\mathbb{N}\right)$ and hence by Theorem \ref{thm:LCAT-over-field-is-ID}, $LCAT\left(\Lambda_s^\mathbb{N}\right)$ is ID.

Assume that $s=pq$ ($p,q\ne1$, they may be equal).
Let $A=p\Lambda_s^\mathbb{N}$ be the set of all sequences  $a\in \Lambda_s^\mathbb{N}$ such that $p|a_n$ for all $n$. $A\ne \Lambda_s^\mathbb{N}$ since $p$ is not invertible in $\Lambda_s$.
Note that $A$ is $LCAT\left(\Lambda_s^\mathbb{N}\right)$-invariant since each coordinate of $\tau (a)$
(for $a\in A, \tau\in LCAT\left(\Lambda_s^\mathbb{N}\right) $ ) is a linear combination of numbers that are divided by $p$.
Hence, $A$ is an infinite non-dense $LCAT\left(\Lambda_s^\mathbb{N}\right)$-invariant subset of $\Lambda_s^\mathbb{N}$, which implies that $LCAT\left(\Lambda_s^\mathbb{N}\right)$ is not ID when $s$ is not a prime.
\hfill $\Box$

\subsection*{Proof of Theorem \ref{thm:LCAT-over-field-is-MC} and Theorem \ref{thm:LCAT-is-MC}}

We prove the maximal commutativity for $LCAT\left(\Lambda_{s}^{\mathbb{N}}\right)$. For the proof of Theorem \ref{thm:LCAT-over-field-is-MC} replace $LCAT\left(\Lambda_{s}^{\mathbb{N}}\right)$ in the following by $LCAT\left(\mathbb{F}_{p^m}^{\mathbb{N}}\right)$.

Recall that $LCAT\left(\Lambda_{s}^{\mathbb{N}}\right)$ is
commutative since it is isomorphic, as a ring, to the commutative ring of polynomials over $\Lambda_{s}$.
Thus, all we have to show is maximality. Let $\tau\in CAT\left(\Lambda_{s}^{\mathbb{N}}\right)$
be a transformation such that $\tau\mu=\mu\tau$ for every $\mu\in LCAT\left(\Lambda_{s}^{\mathbb{N}}\right)$.
Given $\mu\in LCAT\left(\Lambda_{s}^{\mathbb{N}}\right)$ let $p_{\mu}\left(x\right)=\sum_{i=0}^{r}c_{i}x^{i}$ be the polynomial associated with it. Without loss of generality assume that both $\mu$ and $\tau$ are of the same radius $r$, hence $\tau\mu$ and $\mu\tau$ are of radius $\le2r$ and we may consider the radius to be exactly $2r$ (the radius in our definition is not necessarily minimal). For any block $a_{[0,2r]}\in\Lambda^{2r+1}_s$,
\begin{eqnarray*}
\tau\mu\left(a_{[0,2r]}\right) & = & \tau\left(\mu\left(a_{[0,r]}\right)\mu\left(a_{[1,r+1]}\right)...\mu\left(a_{[r,2r]}\right)\right)\\
 & = & \tau\left(\left(\sum\limits _{i=0}^{r}c_{i}a_{i}\right)\left(\sum\limits _{i=0}^{r}c_{i}a_{i+1}\right)...\left(\sum\limits _{i=0}^{r}c_{i}a_{i+r}\right)\right)\end{eqnarray*}
and
\begin{eqnarray*}
\mu\tau\left(a_{[0,2r]}\right) & = & \mu\left(\tau\left(a_{[0,r]}\right)\tau\left(a_{[1,r+1]}\right)...\tau\left(a_{[r,2r]}\right)\right)\\
 & = & \sum\limits _{i=0}^{r}c_{i}\tau\left(a_{[i,i+r]}\right)\end{eqnarray*}
From commutativity
\begin{equation}
\sum\limits _{i=0}^{r}c_{i}\tau\left(a_{[i,i+r]}\right)=\tau\left(\left(\sum\limits _{i=0}^{r}c_{i}a_{i}\right)\left(\sum\limits _{i=0}^{r}c_{i}a_{i+1}\right)\dots \left(\sum\limits _{i=0}^{r}c_{i}a_{i+r}\right)\right)\label{eq:LCA-commutative}\end{equation}
for any choice of $c_{i}\in\Lambda_{s}$ (for $i=0,...,r$) and for
all $a_{[0,2r]}\in\Lambda_{s}^{2r+1}$.
Let $c\in\Lambda_{s}$. Substituting $c_{i}=\begin{cases}
c & \, i=0\\
0 & \, i>0\end{cases}$ in \eqref{eq:LCA-commutative} yields
\[
c\tau\left(a_{[0,r]}\right)=\tau\left(\left(ca_{0}\right) \left(ca_{1}\right)...\left(ca_{r}\right)\right)=\tau\left(ca_{[0,r]}\right)\]
 for all $a_{[0,r]}\in\Lambda_{s}^{r+1}$, which asserts that $\tau$ is
homogeneous.
To show that $\tau$ is additive, we need to verify that for any $x_{0}...x_{r}$,
\[
\tau\left(x_{0}x_{1}...x_{r}\right)=\sum_{i=0}^{r}x_{i}\tau\left(e^{i}\right)\]
where $\left\{ e^{i}\right\}_{i=0}^{r} $ is the standard basis of $\Lambda_{s}^{r+1}$.
Fix $x_{[0,r]}\in\Lambda_{s}^{r+1}$. Substituting $c_{i}=x_{r-i}$
and $a_{i}=\begin{cases}
0 & \, i\ne r\\
1 & \, i=r\end{cases}$ in \eqref{eq:LCA-commutative} yields

\[
x_{r}\tau\left(0...1\right)+x_{r-1}\tau\left(0...10\right)+...+x_{0}\tau\left(1...0\right)=\tau\left(x_{0}x_{1}...x_{r}\right)\]

Note that the left hand side is equal to $\sum_{i=0}^{r}x_{i}\tau\left(e^{i}\right)$, therefore $\tau$ is additive, which proves that $\tau$ is linear, namely $\tau\in LCAT\left(\Lambda_{s}^{\mathbb{N}}\right)$.
\hfill $\Box$

\section{\label{sec:Two-sided-shift}Semigroups over one sided and two sided shift spaces}

So far we dealt only with one sided shift spaces, $\Lambda_{s}^{\mathbb{N}}$.
A cellular automaton transformation over a two sided shift space $\Lambda_{s}^{\mathbb{Z}}$
is a continuous function $\Lambda_{s}^{\mathbb{Z}}\to\Lambda_{s}^{\mathbb{Z}}$
which commutes with the left shift $\sigma$. Similar to the description
of a transformation over $\Lambda_{s}^{\mathbb{N}}$ by a block
map, there exists a two sided version of the Curtis-Hedlund-Lyndon
Theorem:
\chl
Let $\tau:\Lambda_{s}^{\mathbb{Z}}\to\Lambda_{s}^{\mathbb{Z}}$, then
$\tau\in CAT\left(\Lambda_{s}^{\mathbb{Z}}\right)$
if and only if there exists some $r\in\mathbb{N}$ and a function $f_{\tau}:\Lambda_{s}^{2r+1}\to\Lambda_{s}$, such that
\[
\tau\left(a\right)_{k}=f\left(a_{k-r}\dots a_{k-1}a_{k}a_{k+1}\dots a_{k+r}\right)\]
 for any $a\in\Lambda_{s}^{\mathbb{Z}}$ and $k\in\mathbb{N}$.
\medbreak
$f_{\tau}$ is called a block map and $r$ is the radius of $\tau$. As in the one sided case, it is often convenient to use the same notation for the block map $f_{\tau}$ and the transformation $\tau$.

Every block map of $\tau\in CAT\left(\Lambda_{s}^{\mathbb{N}}\right)$
is a valid block map for a transformation in $CAT\left(\Lambda_{s}^{\mathbb{Z}}\right)$. Under this correspondence we may regard $\tau$ as being contained in $CAT\left(\Lambda_{s}^{\mathbb{Z}}\right)$ and we may say that $CAT\left(\Lambda_{s}^{\mathbb{N}}\right)\subset CAT\left(\Lambda_{s}^{\mathbb{Z}}\right)$.
In particular, any semigroup $\Sigma\subset CAT\left(\Lambda_{s}^{\mathbb{N}}\right)$
can also be considered as a semigroup $\Sigma\subset CAT\left(\Lambda_{s}^{\mathbb{Z}}\right)$.

For $\Sigma\subset CAT\left(\Lambda_{s}^{\mathbb{Z}}\right)$ such
that $\sigma\in\Sigma$, denote $\Sigma_{+}=\Sigma\cap CAT\left(\Lambda_{s}^{\mathbb{N}}\right)$.
$\Sigma_+$ contains all the transformations in $\Sigma$ which can act also on $CAT\left(\Lambda_{s}^{\mathbb{N}}\right)$. Observe that for each $\tau \in \Sigma \backslash \Sigma_+$, if we compose it with $\sigma^r$ where $r$ is the radius of $\tau$, we get an element in $\Sigma_+$.
Note that $\Sigma_{+}$ is a sub-semigroup of $CAT\left(\Lambda_{s}^{\mathbb{N}}\right)$ (when it is non empty).

One may think of a cellular automaton as an action of a machine that reads the first $r+1$ symbols ($a_0\dots a_{r}$), assigns to them an output according to its block map and inserts it in the first coordinate. Then, the machine moves to the next $r+1$ symbols ($a_1\dots a_{r+1}$) and inserts the corresponding output in the second coordinate and so on. The machines $\tau$ and $\sigma^{z}\tau$ ($z\in\mathbb{Z}$) use the same block map and differ in the location where they insert the output. This perspective suggests to view a cellular automaton transformation as being determined by its block map rather than where the precise location of the output is. This motivates the following:

\begin{theo}
\label{thm:one->two-ID}Let $\Sigma\subset CAT\left(\Lambda_{s}^{\mathbb{N}}\right)$
be an ID semigroup, then $\left\langle \Sigma,\sigma^{-1}\right\rangle \subset CAT\left(\Lambda_{s}^{\mathbb{Z}}\right)$
is an ID semigroup.
\end{theo}

\begin{theo} \label{thm:one->two-MC}
Let $\Sigma \subset CAT\left( \Lambda_s^\mathbb{N} \right)$ be a maximal commutative semigroup, then $\left\langle \Sigma,\sigma^{-1}\right\rangle$ is maximal commutative in $CAT\left( \Lambda_s^\mathbb{Z} \right)$.
\end{theo}

\begin{theo}
\label{thm:two->one-ID}
Let $\Sigma\subset CAT\left(\Lambda_{s}^{\mathbb{Z}}\right)$
be an ID semigroup, then $\left\langle \Sigma,\sigma\right\rangle_{+} \subset CAT\left(\Lambda_{s}^{\mathbb{N}}\right)$
is an ID semigroup.
\end{theo}

\begin{theo} \label{thm:two->one-MC}
Let $\Sigma \subset CAT\left( \Lambda_s^\mathbb{Z} \right)$ be a maximal commutative semigroup then $\Sigma_{+}$ is maximal commutative in $CAT\left( \Lambda_s^\mathbb{N} \right)$.
\end{theo}

Denote by $\pi_{+}:\Lambda_{s}^{\mathbb{Z}}\to\Lambda_{s}^{\mathbb{N}}$
the projection on the positive coordinates:
$$\pi_{+}\left(\left\{ a_{n}\right\} _{-\infty}^{\infty}\right)=\left\{ a_{n}\right\} _{0}^{\infty}.$$

Note that these statements are sharp in the following sense: In Theorem \ref{thm:one->two-ID}, if we delete $\sigma^{-1}$, then  $\Sigma\subset CAT\left(\Lambda_{s}^{\mathbb{Z}}\right)$ is not an ID semigroup (the set of all sequences such that $\pi_+\left(a\right)$ is 1-periodic is invariant). Clearly, it is necessary to add $\sigma^{-1}$ in Theorem \ref{thm:one->two-MC}. The statement of Theorem \ref{thm:two->one-ID} is meaningless without adding $\sigma$ since $\Sigma_+$ may be empty.

\begin{definition}
A sequence $a\in\Lambda_{s}^{\mathbb{Z}}$ is
\textbf{$c$-periodic} if for any $z\in\mathbb{Z},$ $a_{z}=a_{z+c}$, where
$c$ is minimal with this property.
\end{definition}

\subsection*{Proof of Theorem \ref{thm:one->two-ID}}

Let $Y\subset \Lambda_{s}^{\mathbb{Z}}$ be an infinite set, and let $w_{-n}\dots w_0 \dots w_n$ be some finite word of length $2n+1$. Define $X=\left\{\sigma^{-k}(y)|y\in Y,k>n\right\}$ and note that $\pi_+\left(X\right)$ is an infinite subset of $\Lambda_s^\mathbb{N}$. By the fact that $\Sigma$ is ID, there exist $\tilde{x}\in\pi_+\left(X\right)$ and $\tau\in \Sigma$ such that $\tau(\tilde{x})_{[0,2n]}= w_{-n}\dots w_0 \dots w_n$ and $\tilde{x}=\sigma^{-k}(y)$ where $k>n$ and $y\in Y$.
That is $\tau\sigma^{-k+n}(y)_{[-n,n]}=w_{-n}\dots w_0 \dots w_n$. Since $\tau\sigma^{-k+n}\in \left\langle \Sigma,\sigma^{-1}\right\rangle$ and $w_{-n}\dots w_0 \dots w_n$ is an arbitrary word, we conclude that $\left\langle \Sigma,\sigma^{-1}\right\rangle$ is ID.
\hfill $\Box$

\subsection*{Proof of Theorem \ref{thm:one->two-MC}}

Let $\tau\in CAT\left(\Lambda_{s}^{\mathbb{Z}}\right)$ with radius
$r$ and consider $\tau\sigma^{r}$. Note that for every $k\in\mathbb{Z}$,
the $k^{\mbox{th}}$ coordinate $\tau\sigma^{r}\left(a\right)_{k}$
is dependent only on $a_{[k,k+2r+1]}$. Thus, $\tau\sigma^{r}$ can act also on the one sided shift space, hence $\tau\sigma^{r}\in CAT\left(\Lambda_{s}^{\mathbb{N}}\right)$
(with radius $\le2r$).

Assume that $\tau$ commutes with all $\Sigma'=\left\langle \Sigma,\sigma^{-1}\right\rangle$,
we need to show that $\tau\in\Sigma'$.
In particular, $\tau$ commutes with $\Sigma$. Let $\mu\in\Sigma$,
then $\tau\mu=\mu\tau$ and hence $\sigma^{r}\tau\mu=\mu\sigma^{r}\tau$. Since $\sigma^{r}\tau\in CAT\left(\Lambda_{s}^{\mathbb{N}}\right)$
and $\Sigma$ is MC, $\sigma^{r}\tau\in\Sigma$,
which implies that $\tau\in\Sigma'ê$.
\hfill $\Box$

\subsection*{Proof of Theorem \ref{thm:two->one-ID}}

Let $X\subset\Lambda_{s}^{\mathbb{N}}$ an infinite set, and let $Y\subset\Lambda_{s}^{\mathbb{Z}}$
be the set $Y=\left\{ y|\forall n\in\mathbb{N},\exists x\in X,m\in\mathbb{N},y_{\left[-n,n\right]}=x_{\left[m,m+2n\right]}\right\} $,
that is, every symmetric word in an element of $Y$, appears as a word in some element of $X$.

First we show that $Y\ne\emptyset$. Given $x\in X$ we construct
$y\in Y$: There exists a symbol $\eta$ that appears infinitely many
times in $x$. Define $y_{0}=\eta$. There exist two symbols
$\zeta$ and $\theta$, such that the word $\zeta\eta\theta$ appears infinitely
many times in $x$. Choose $y_{\left[-1,1\right]}=\zeta\eta\theta$. Continue
with this process to get that $y\in Y$. In particular, $Y\ne\emptyset$
and for every $x\in X$ and for arbitrary large length $m$, there
exists this $y$ that shares a word of length $m$ with $x$.

 Assume that $\left|Y\right|=\infty$. Since $\Sigma$ is ID,
$\Sigma\left(Y\right)=\left\{ \tau\left(y\right)\right\} _{\tau\in\Sigma,y\in Y}$
is dense, that is, any arbitrarily large symmetric word appears in some
element of $\Sigma\left(Y\right)$. Let $w_{0}\dots w_{k}\in\Lambda_{s}^{k+1}$
be a word. Find $\tau$ and $y$ such that $\tau\left(y\right)_{\left[0,k\right]}=w_{0}\dots w_{k}$
and denote by $r$ the radius of $\tau$. By the definition of $Y$,
there exists $x\in X$ and a length $m>0$, such that $y_{\left[r,k+3r\right]}=x_{\left[m,m+k+2r\right]}$. Note that $\tau\in\Sigma$ implies that $\tau\sigma^{r}\in\left\langle \Sigma,\sigma\right\rangle _{+}$,
and $\tau\sigma^{r}\left(x\right)_{\left[m,m+k\right]}=\tau\sigma^{r}\left(y\right)_{[r,k+r]}=\tau\left(y\right)_{\left[0,k\right]}$.
Hence $\tau\sigma^{r+m}\left(x\right)_{\left[0,k\right]}=w_{0}\dots w_{k}.$
 Since $\tau\sigma^{r+m}\in\left\langle \Sigma,\sigma\right\rangle _{+}$
and the word $w_{0}\dots w_{k}$ is arbitrary, we conclude that $\left\langle \Sigma,\sigma\right\rangle _{+}$
is ID.

Therefore, it is enough to show that $\left|Y\right|=\infty$. Assume by contradiction
that $\left|Y\right|<\infty$. Since by definition $Y$ is $\sigma$-invariant,
it consists only of periodic sequences.

Suppose that there exists $y\in Y$ such that for arbitrary large
$k$, there exists $x^{\left(k\right)}\in X$ and $n>k$ with $y_{\left[-k,-1\right]}=x_{\left[-k+n,n-1\right]}^{\left(k\right)}$
but $y_{0}\ne x_{n}^{\left(k\right)}$. We can assume without loss
of generality that $x_{n}^{\left(k\right)}$ is the same symbol,
for all $k$. Thus there exists $y'\in Y$
such that $y'_{(\infty,-1]}=y_{(\infty,-1]}$ but $y'_{0}\ne y_{0}$.
Now either $y$ or $y'$ is non periodic which is contradiction.

Hence, for every $y\in Y$ there exists a length $M\left(y\right)$
such that if $y_{\left[-k,-1\right]}=x_{\left[-k+n,n-1\right]}$ (for
$x\in X,k>M\left(y\right),n>k$) then $y_{0}=x_{n}$. Denote by $M=\max_{y\in Y}\left\{ M\left(y\right)\right\}$.
If $x\in X$ shares a word of length $M$ with some element
in $Y$ then they are right asymptotic ($\exists n\in\mathbb{N},z\in\mathbb{Z}$ such that $x_{[n,\infty)}=y_{[z,\infty)}$). But we saw that every $x\in X$
shares arbitrarily long words with some $y\in Y$ so every $x\in X$ is right
asymptotic to some $y\in Y$.

Now, since $Y$ contains only periodic sequences, there are infinitely
many sequences $\left\{ x^{\left(n\right)}\right\} \subset X$ that
are eventually periodic with the same period. Moreover, we can assume that they are all right asymptotic. $x^{\left(n\right)}$
is $b^{\left(n\right)},c$-eventually periodic, where $b^{\left(n\right)}\to\infty$.
By the minimality of the parameters, the sequence $x_{[b\left(n\right)-1,\infty)}^{\left(n\right)}$
is not periodic (while $x_{[b\left(n\right),\infty)}^{\left(n\right)}$
is periodic). Define $y\in\Lambda_{s}^{\mathbb{Z}}$ by $y_{[0,\infty)}=x_{[b\left(n\right),\infty)}^{\left(n\right)}$.
Consider the symbols appearing in $\{x_{b\left(n\right)-1}^{\left(n\right)}\}$ (where $b(n)>0$).
Without loss of generality all of these are the same symbol $\eta$
and define $y_{-1}=\eta$. Consider the symbols appearing in $\{x_{b(n)-2}^{(n)}\}$
for such $x^{\left(n\right)}$ where $b\left(n\right)>1$. One of
the symbols appears infinitely many times and set $y_{-2}$ to be
that symbol. Continue with this process to get $y\in\Lambda_{s}^{\mathbb{Z}}$.
By the construction, $y\in Y$ and it is non periodic which is contradiction.
We conclude that $Y$ is infinite and the proof is completed.

\hfill $\Box$

\subsection*{Proof of Theorem \ref{thm:two->one-MC}}

Let $\tau\in CAT\left(\Lambda_{s}^{\mathbb{N}}\right)$ such that $\tau$
commutes with all $\Sigma_{+}$.

First we show that $\tau$, as a transformation in $CAT\left(\Lambda_{s}^{\mathbb{Z}}\right)$,
commutes with all $\Sigma$. Let $\mu\in\Sigma$ be a transformation with radius $r$. Since $\Sigma$ is MC it must contain $\sigma$ and thus $\mu\sigma^{r}\in\Sigma_{+}$. $\tau$ commutes with all the elements of $\Sigma_{+}$, in particular $\mu\sigma^{r}\tau=\tau\mu\sigma^{r}$. $\sigma^r$ commutes with both $\tau,\mu$, and $\sigma^r$ is invertible, hence $\mu\tau=\tau\mu$. The MC of $\Sigma$ implies that $\tau\in\Sigma$. Thus $\tau\in\Sigma\cap CAT\left(\Lambda_{s}^{\mathbb{N}}\right)=\Sigma_{+}$.
\hfill $\Box$

\section{\label{sec:The-Action-of-Aut}The Action of the Automorphism Group }

Denote the group of all the automorphisms of the dynamical system
$\left(\Lambda_{s}^{\mathbb{Z}},\sigma\right)$ by $Aut\left(\Lambda_{s}^{\mathbb{Z}}\right)$.
$Aut\left(\Lambda_{s}^{\mathbb{Z}}\right)$ consists of all the cellular automata transformations which are bijective functions.

The automorphism group is {}``large'' in the algebraic sense.
Hedlund \cite{hedlund1969endomorphisms} showed that $Aut\left(\Lambda_{s}^{\mathbb{Z}}\right)$ contains a copy of every finite group, and Boyle Lind and Rudolph \cite{boyle1988automorphism} showed that is contains free groups. Recently
$Aut\left(\Lambda_{s}^{\mathbb{Z}}\right)$ has been extensively studied
and a rich algebraic theory has been developed. The automorphism group is studied usually for automorphisms of subshift of finite type, as well as for automorphism group of multidimensional shifts of finite type (see \cite{hochman2010automorphism}).

The automorphism group is {}``large'' in the dynamical sense as well.
Boyle, Lind and Rudolph \cite{0881.58035} showed that $\Lambda_{s}^{\mathbb{Z}}$ does not contain a proper infinite closed set which is $Aut\left(\Lambda_{s}^{\mathbb{Z}}\right)$-invariant. This means in our terms that $Aut\left(\Lambda_{s}^{\mathbb{Z}}\right)$ is an ID semigroup. Their theorem is more general and holds for automorphism group of any mixing shift of finite type (see \cite{0881.58035}). Consider the case of the full shift space when $s$ is not a power of prime. A stronger version of BLR Theorem, for that particular case, can be derived from previous results in this article.

In section \ref{sec:Arithmetic-semi-groups} we defined $\mu_{u}\in CAT\left(\Lambda_{s}^{\mathbb{N}}\right)$ ($u|s^n$ for some $n$). As already mentioned (section \ref{sec:Two-sided-shift}) $\mu_{u}$ can be viewed
as a transformation in $CAT\left(\Lambda_{s}^{\mathbb{Z}}\right)$.
As such, $\mu_{u}\in CAT\left(\Lambda_{s}^{\mathbb{Z}}\right)$ is both onto and injective (note that $\left(\mu_{u}\right)^{-1}=\sigma^{-1}\mu_{\frac{s}{u}}$)). Clearly, as a transformation acting on a one sided shift space ($\mu_{p}\in CAT\left(\Lambda_{s}^{\mathbb{N}}\right)$), is onto but not injective.

\begin{theo}
$\left\langle \mu_{p},\sigma,\sigma^{-1}\right\rangle $ is an ID semigroup of $CAT\left(\Lambda_{s}^{\mathbb{Z}}\right)$ where $p|s$ is prime and $s$ is not a power $p$.

In particular, $Aut\left(\Lambda_{s}^{\mathbb{Z}}\right)$ is an ID semigroup where $s$ is not a power of prime.
\end{theo}
\begin{proof}
Let $s=pq$ where $p$ is a prime that does not divide $q$. From Theorem \ref{thm:multi-is-ID}, $\left\langle \mu_{p},\sigma\right\rangle$ is an ID semigroup of $CAT\left(\Lambda_{s}^{\mathbb{N}}\right)$.
Hence, by Theorem \ref{thm:one->two-ID}, $\left\langle \mu_{p},\sigma,\sigma^{-1}\right\rangle $
is an ID semigroup of $CAT\left(\Lambda_{s}^{\mathbb{Z}}\right)$.
$\mu_{p}\in Aut\left(\Lambda_{s}^{\mathbb{Z}}\right)$
and clearly $\sigma,\sigma^{-1}\in Aut\left(\Lambda_{s}^{\mathbb{Z}}\right)$,
thus, $Aut\left(\Lambda_{s}^{\mathbb{Z}}\right)$ is an ID semigroup as it contains the ID sub-semigroup $\left\langle \mu_{p},\sigma,\sigma^{-1}\right\rangle $.
\end{proof}

\section{\label{sec:ID-and-Maximal}ID and Maximal Commutativity}

Let $\Omega$ be a shift space, $\Sigma\subset CAT\left(\Omega\right)$ a semigroup and let $\Sigma'$ be a sub-semigroup of $\Sigma$. Clearly, any $\Sigma$-invariant subset of $\Omega$ is in particular $\Sigma'$-invariant. Therefore, any semigroup which contains an ID semigroup is also ID.

On the other hand, the MC property in not preserved under inclusion (adding new elements will break down the commutativity) or by moving to smaller semigroup (removing elements will break the maximality down).

A semigroup which is ID is {}``large'' in a dynamical sense, while the MC property indicates {}``largeness'' in an algebraic sense.

However, neither of the properties implies the other: the semigroup $\left\langle \mu_{2},\sigma\right\rangle \subset CAT\left(\Lambda_{10}^\mathbb{N}\right) $ (section \ref{sec:Arithmetic-semi-groups}) is ID but not MC. While $LCAT\left(\Lambda_s^\mathbb{N}\right)$ is not ID when $s$ is not prime, but it is MC for any $s$.

Although ID and MC are properties that indicate {}``largeness'' we saw a semigroup generated by only two transformations which is ID. In the following, we construct a semigroup generated by two transformations which is MC (and not ID).

Consider the case $s=2$. As in section \ref{sec:Algebraic-Semigroups} one can treat $\Lambda_{2}$ as the field $\frac{\mathbb{Z}}{2\mathbb{Z}}$. Let $\tau\in CAT\left(\Lambda_{2}^{\mathbb{Z}}\right)$ be defined by
\[ \tau \left(a\right)_k=a_{k}+\prod\limits _{i=1}^{r}\left(a_{k+i}+\delta_{i}\right)\]
where addition and multiplication are defined in $\frac{\mathbb{Z}}{2\mathbb{Z}}$, $\delta_{i}\in\frac{\mathbb{Z}}{2\mathbb{Z}}$ ($i=1,...,r$) and the radius $r\ge2$.

Define the \textbf{least period of $\tau$} by the minimal integer $m$ such that $\delta_{i}=\delta_{i+m}$ for $i=1,2,..,r-m$.

Coven, Hedlund and Rhodes \cite{coven1979commuting} solved the {}``Commuting Block Maps Problem'' for several classes of block maps, in particular for the block maps of the form \textbf{that defines $\tau$} above. They showed that the set transformations commuting with $\tau$ depends on whether $m\le\frac{r}{2}$ or not.
We rephrase their solution for transformations of the form of $\tau$ with $m\le\frac{r}{2}$ using the terminology of the present article:
\spt
Let $\tau\in CAT\left(\Lambda_{2}^{\mathbb{Z}}\right)$
be defined by
\[
\tau\left(a\right)_k=a_{k}+\prod\limits _{i=1}^{r}\left(a_{k+i}+\delta_{i}\right)\]
where the least period of $\tau \le\frac{r}{2}$. Then $\left\langle \tau,\sigma, {\sigma}^{-1}, Id\right\rangle \subset CAT\left(\Lambda_{2}^{\mathbb{Z}}\right)$
is MC.
\medbreak

Consider $\tau$ with the following parameters: $r=2$ and $\delta_{1}=\delta_{2}=0$. The resulting transformation is $\tau\left(a\right)_{k}=a_{k}+a_{k+1}a_{k+2}$.

$\tau$ {}``changes'' the $k^{\mbox{th}}$ coordinate
if and only if $a_{k+1}=a_{k+2}=1$. For each $i\in\mathbb{Z}$, define $e^{i}\in\Lambda_{2}^{\mathbb{Z}}$
to be the sequence with 1 in the $i^{\mbox{th}}$ coordinate and 0 elsewhere. The consecutive pair $11$ does not appear in $e^{i}$ and therefore $\tau\left(e^{i}\right)=e^{i}$ for all $i\in\mathbb{Z}$. Define $A=\overline{\left\{ e^{i}\right\} _{i\in\mathbb{Z}}}=\left\{ e^{i}\right\} _{i\in\mathbb{Z}} \cup \ \left\{ 0\right\}$ (where $\left\{ 0\right\} =000\dots$). $A\subset \Lambda_{2}^{\mathbb{Z}}$ is an infinite closed $\left\langle \tau,\sigma\right\rangle $-invariant
proper subset of $\Lambda_{2}^{\mathbb{Z}}$, which asserts that $\left\langle \tau,\sigma\right\rangle $
is not an ID semigroup.
\medbreak

By Theorem \ref{thm:two->one-MC} and the Small Period Theorem we have that $\left\langle \tau,\sigma, Id \right\rangle$ is MC in $CAT\left(\Lambda_2^{\mathbb{Z}}\right)ê$. This is an example of an MC semigroup over the one-sided shift space which is generated by $3$ elements only.
Clearly, since every transformation commutes with $\sigma$ and $Id$, therefore, $\left\langle \tau,\sigma, Id\right\rangle \subset CAT\left(\Lambda_{2}^{\mathbb{Z}}\right)$ is of minimal number of generators for a semigroup having the MC property over $\Lambda_2$.

Let $\Sigma\subset CAT\left(\Lambda_s^\mathbb{N}\right)$ be an MC semigroup. For any $s$, the minimal number of generators of $\Sigma$ is at least $3$. For $s=2$ we saw that this number is actually $3$.
\begin{problem}
What is the minimal number of generators for a semigroup having the MC property for any $s$?
\end{problem}

For $s$ which is not a power of a prime, we constructed a semigroup in $CAT\left(\Lambda_s^\mathbb{N}\right)$ generated by 2 elements which is ID. A semigroup $\Sigma$ is called cyclic if $\Sigma=\left\langle \tau \right\rangle$ for some $\tau$.
\begin{problem}
Is it true that a cyclic semigroup over  $\Lambda_s^\mathbb{N}$ ($s\ge 2 $) cannot have the ID property?

What is the minimal number of generators for an ID semigroup over $\Lambda_{s}^\mathbb{N}$ for general $s$?
\end{problem}

By the theorems in section \ref{sec:Two-sided-shift} one can see that the questions above, do not change significantly, when asking these questions about a semigroup over a two sided shift space.

Although the ID property does not always follow from MC, it may be so for certain MC semigroups. Therefore we pose the following interesting question.
\begin{problem}
Let $\Omega$ be a shift space and $\Sigma\subset CAT\left(\Omega\right)$
a maximal commutative semigroup of cellular automata transformations. Are there any
conditions on $\Sigma$ which imply that $\Sigma$ has the ID property?
\end{problem}

\section{Acknowledgements}
This article is based on an M.Sc. thesis submitted to the Hebrew University, Jerusalem. I would like to thank my advisor, Prof. Hillel Furstenberg, for introducing me to the area and for sharing with me the joy of doing mathematics. I would like also to thank Dr. Uri Shapira and Dr. Jehuda Hartman for their help. Special thanks are due to the referee for his helpful suggestions and improvements in several proofs.

\bibliographystyle{plain}
\bibliography{Thesis}

\begin{bibdiv}
\begin{biblist}

\bib{berend1983multi}{article}{
      author={Berend, D.},
       title={Multi-invariant sets on tori},
        date={1983},
        ISSN={0002-9947},
     journal={Transactions of the American Mathematical Society},
      volume={280},
      number={2},
       pages={509\ndash 532},
}

\bib{berend1984multi}{article}{
      author={Berend, D.},
       title={Multi-invariant sets on compact abelian groups},
        date={1984},
        ISSN={0002-9947},
     journal={Transactions of the American Mathematical Society},
      volume={286},
      number={2},
       pages={505\ndash 535},
}

\bib{0881.58035}{article}{
      author={Blanchard, F.},
      author={Host, B.},
      author={Maass, A.},
       title={{Representation by an automaton of continuous functions of the
  torus. (Repr\'esentation par automate de fonctions continues de tore.)}},
    language={French},
        date={1996},
}

\bib{blanchard1997dynamical}{article}{
      author={Blanchard, F.},
      author={Maass, A.},
       title={Dynamical properties of expansive one-sided cellular automata},
        date={1997},
        ISSN={0021-2172},
     journal={Israel Journal of Mathematics},
      volume={99},
      number={1},
       pages={149\ndash 174},
}

\bib{boyle1988automorphism}{article}{
      author={Boyle, M.},
      author={Lind, D.},
      author={Rudolph, D.},
       title={{The automorphism group of a shift of finite type}},
        date={1988},
        ISSN={0002-9947},
     journal={Transactions of the American Mathematical Society},
      volume={306},
      number={1},
       pages={71\ndash 114},
}

\bib{ceccherini2006garden}{article}{
      author={Ceccherini-Silberstein, T.},
      author={Coornaert, M.},
       title={The garden of eden theorem for linear cellular automata},
        date={2006},
     journal={Ergodic Theory and Dynamical Systems},
      volume={26},
      number={01},
       pages={53\ndash 68},
}

\bib{ceccherini2010cellular}{book}{
      author={Ceccherini-Silberstein, T.G.},
      author={Coornaert, M.},
       title={Cellular automata and groups},
   publisher={Springer Verlag},
        date={2010},
}

\bib{coven1979commuting}{article}{
      author={Coven, E.M.},
      author={Hedlund, GA},
      author={Rhodes, F.},
       title={{The commuting block maps problem}},
        date={1979},
        ISSN={0002-9947},
     journal={Transactions of the American Mathematical Society},
      volume={249},
      number={1},
       pages={113\ndash 138},
}

\bib{ferguson1962some}{article}{
      author={Ferguson, J.D.},
       title={{Some properties of mappings on sequence spaces}},
        date={1962},
     journal={Dissertation, Yale University},
}

\bib{furstenberg1967disjointness}{article}{
      author={Furstenberg, H.},
       title={{Disjointness in ergodic theory, minimal sets, and a problem in
  Diophantine approximation}},
        date={1967},
        ISSN={1432-4350},
     journal={Theory of Computing Systems},
      volume={1},
      number={1},
       pages={1\ndash 49},
}

\bib{hedlund1969endomorphisms}{article}{
      author={Hedlund, GA},
       title={{Endomorphisms and automorphisms of the shift dynamical system}},
        date={1969},
        ISSN={1432-4350},
     journal={Theory of Computing Systems},
      volume={3},
      number={4},
       pages={320\ndash 375},
}

\bib{hochman2010automorphism}{article}{
      author={Hochman, M.},
       title={On the automorphism groups of multidimensional shifts of finite
  type},
        date={2010},
     journal={Ergodic Theory and Dynamical Systems},
      volume={30},
      number={03},
       pages={809\ndash 840},
}

\bib{ito1983linear}{article}{
      author={Ito, M.},
      author={Osato, N.},
      author={Nasu, M.},
       title={Linear cellular automata over zm},
        date={1983},
     journal={Journal of Computer and System sciences},
      volume={27},
      number={1},
       pages={125\ndash 140},
}

\bib{katok1996invariant}{article}{
      author={Katok, A.},
      author={Spatzier, R.J.},
       title={{Invariant measures for higher-rank hyperbolic abelian actions}},
        date={1996},
        ISSN={0143-3857},
     journal={Ergodic Theory and Dynamical Systems},
      volume={16},
      number={04},
       pages={751\ndash 778},
}

\bib{lindenstrauss2005invariant}{article}{
      author={Lindenstrauss, E.},
       title={{Invariant Measures for Multiparameter Diagonalizable Algebraic
  Actions-A Short Survey}},
organization={ASM International, Incorporated},
        date={2005},
       pages={247},
}

\bib{rudolph1990x2}{article}{
      author={Rudolph, D.J.},
       title={{x2 and x3 invariant measures and entropy}},
        date={1990},
     journal={Ergodic Theory and Dynamical Systems},
      volume={10},
       pages={395\ndash 406},
}

\end{biblist}
\end{bibdiv}

\end{document}